\documentclass[11pt,reqno]{amsart}
\usepackage[margin=3cm]{geometry}
\usepackage{color}               
\usepackage{faktor}
\usepackage{tikz}
\usepackage{tikz-cd}
\usepackage{spverbatim}
\usepackage{float}
\usepackage[pdftex,hyperfootnotes=false,colorlinks=false,hypertexnames=false]{hyperref}

\usepackage{amsmath,amssymb,anyfontsize,enumerate,stmaryrd,mathtools, thmtools,tikz,txfonts,nccmath}
\usepackage[oldstyle]{libertine}%
\usepackage[T1]{fontenc}
\usepackage[final]{microtype}
\usepackage[utf8]{inputenc}
\usepackage{csquotes}
\makeatletter
\renewcommand*\libertine@figurestyle{LF}
\makeatother
\usepackage[libertine,libaltvw,liby]{newtxmath}
\makeatletter
\renewcommand*\libertine@figurestyle{OsF}
\makeatother
\usepackage[noerroretextools,backend=biber,style=alphabetic,maxalphanames=5,giveninits,sorting=nyt,%
maxbibnames=99,doi=false,isbn=false,url=false,eprint=false]{biblatex}
\usepackage{tikz-cd}
\usepackage[noabbrev,capitalise]{cleveref}
\usepackage{autonum}
\theoremstyle{plain}
    \newtheorem{theorem}{Theorem}
    \newtheorem{construction/theorem}[theorem]{Construction/Theorem}
    \newtheorem{corollary}[theorem]{Corollary}
    \newtheorem{lemma}[theorem]{Lemma}
    \newtheorem{proposition}[theorem]{Proposition}
\theoremstyle{definition}
    \newtheorem{remark}[theorem]{Remark}
    \newtheorem{example}[theorem]{Example}
    \newtheorem{definition}[theorem]{Definition}
    \newtheorem{construction}[theorem]{Construction}

\DeclareMathOperator{\Aut}{Aut}
\DeclareMathOperator{\trop}{trop}
\DeclareMathOperator {\WC}{WC}

\title[Tropical twisted Hurwitz numbers]{Twisted Hurwitz numbers: Tropical and polynomial structures}
\author[M.~A.~Hahn]{Marvin Anas Hahn}
\address{M.~A.~Hahn: School of Mathematics 17, Westland Row, Trinity College Dublin, Dublin 2, Ireland}
\email{hahnma@tcd.ie}
\author[H.~Markwig]{Hannah Markwig}
\address{H.~Markwig: Universität Tübingen, Fachbereich Mathematik, Auf der Morgenstelle 10, 72076 Tübingen, Germany}
\email{hannah@math.uni-tuebingen.de}


\begin{document}

\maketitle
\begin{abstract}

Hurwitz numbers count covers of curves satisfying fixed ramification data. Via monodromy representation, this counting problem can be transformed to a problem of counting factorizations in the symmetric group. This and other beautiful connections make Hurwitz numbers a longstanding active research topic.
In recent work \cite{chapuy2020non}, a new enumerative invariant called \textit{$b$-Hurwitz number} was introduced, which enumerates non-orientable branched coverings. For $b=1$, we obtain twisted Hurwitz numbers which were linked to surgery theory in \cite{BF21} and admit a representation as factorisations in the symmetric group. In this paper, we derive a tropical interperetation of twisted Hurwitz numbers in terms of tropical covers and study their polynomial structure.
\end{abstract}
\section{Introduction}

Hurwitz numbers are enumerations of branched morphisms between Riemann surfaces with fixed numerical data. They go back to work by Adolf Hurwitz in the 1890s \cite{hurwitz1892algebraische} and are now important invariants in enumerative geometry. 
They admit various equivalent descriptions in the language of different areas of mathematics, e.g., as shown by Hurwitz in his above-mentioned work, they can be computed by an enumeration of transitive factorisations in the symmetric group. This equivalence gives rise to a deep connection between Hurwitz theory and the representation theory of the symmetric group; it will also play a key role in the present work. Moreover, Hurwitz numbers are closely related to the algebraic topology underlying Riemann surfaces, since they turn out to be \textit{topological invariants}. While the theory of Hurwitz numbers has been dormant for most of the 20th century, the close relationship between Hurwitz and Gromov-Witten theory discovered in the 1990s has rekindled interest in these enumerative invariants and led to several exciting developments.\vspace{\baselineskip}

\subsection{Hurwitz numbers, Gromov--Witten theory and variants.}
When studying relations between Hurwitz numbers and Gromov--Witten theory, certain classes of Hurwitz numbers with particularly well--behaved structures take center stage. Among these classes are so-called \textit{double Hurwitz numbers}, which are defined as follows.

\begin{definition}[Double Hurwitz numbers]\label{def-hur}
Let $g\ge0$ be a non-negative integer, $n>0$ a positive integer and $\mu,\nu$ partitions of $n$. Moreover, we fix $p_1,\dots,p_b\in\mathbb{P}^1$, where $b=2g-2+\ell(\mu)+\ell(\nu)$. Then, we define a cover of type $(g,\mu,\nu)$ to be a map $f\colon S\to\mathbb{P}^1$, such that
\begin{itemize}
    \item $S$ is a connected Riemann surface of genus $g$;
    \item the ramification profile of $0$ is $\mu$;
    \item the ramification profile of $\infty$ is $\nu$;
    \item the ramification profile of $p_1,\dots,p_b$ is $(2,1\dots,1)$.
\end{itemize}
Two covers $f\colon S\to\mathbb{P}^1$ and $f'\colon S'\to\mathbb{P}^1$ are called equivalent if there exists a homeomorphism $g\colon S\to S'$, such that $f= f'\circ g$.\\
Then, we define \textbf{double Hurwitz numbers} as
\begin{equation}
    h_g(\mu,\nu)=\sum_{[f]}\frac{1}{|\mathrm{Aut}(f)|},
\end{equation}
where the sum runs over all equivalence classes of covers of type $(g,\mu,\nu)$.\\
When $\nu=(1,\dots,1)$, we call $h_g(\mu,\nu)$ a \textbf{single Hurwitz number} and denote it by $h_g(\mu)$.
\end{definition}

At the core of the relationship between double Hurwitz numbers and Gromov--Witten theory is a polynomial structure in the prescribed ramification data of these enumerative invariants. First discovered in the seminal work of Goulden, Jackson and Vakil in \cite{GJV05}, double Hurwitz numbers exhibit a \textit{piecewise polynomial} behaviour. More precisely, we consider the space \begin{equation}
    \mathcal{H}_{m.n}\coloneqq\{(\mu,\nu)\in\mathbb{N}^m\times\mathbb{N}^n\mid \sum \mu_i=\sum \nu_j\}
\end{equation}
of partitions $(\mu,\nu)$ of fixed lengths $m,n$ and of the same size. For fixed $g$, we consider the map
\begin{align}
    h_g\colon \mathcal{H}_{m,n}&\to\mathbb{Q}\\
    (\mu,\nu)&\mapsto h_g(\mu,\nu)
\end{align}
which parametrises double Hurwitz numbers.

The authors of \cite{GJV05} showed that there exists a hyperplane arrangement $\mathcal{R}_{m,n}$ in $\mathcal{H}_{m,n}$ (called the resonance arrangement), such that the map $h_g$ restricted to each connected component (called \textit{chamber}) of $\mathcal{H}_{m,n}\backslash\mathcal{R}_{m,n}$ may be represented as a polynomial in the entries of $\mu$ and $\nu$. In \cite{SSV08, CJM11, Joh15}, the natural question of how the polynomials differ from chamber to chamber was studied. It was observed that there is a recursive structure in the sense that this difference can be expressed by double Hurwitz numbers with smaller input data. This is called a \textit{wall-crossing formula}.\\
We want to highlight the work in \cite{CJM11}, in which a graph theoretic approach towards the polynomiality of double Hurwitz numbers was established. The key technique in this paper revolves around the field of \textit{tropical geometry}. Tropical geometry is a relatively new field of mathematics, which may be described as a combinatorial shadow of algebraic geometry. The tropical geometry perspective allows to degenerate algebraic curves to certain metric graphs that are called \textit{tropical curves}. In this manner, branched morphisms between Riemann surfaces are \textit{tropicalised} to maps between tropical curves that are called \textit{tropical covers}. Motivated by this point of view, a combinatorial interpretation of double Hurwitz numbers in terms of tropical covers was derived in \cite{CJM10}, which laid the groundwork for the analysis of the polynomial behaviour of double Hurwitz numbers undertaken in \cite{CJM11}. In particular, by proceeding along an intricate combinatorial analysis of tropical covers in different chambers, the authors of \cite{CJM11} were able to derive the desired wall-crossing structure.\\
In the past years, several variants of Hurwitz numbers have appeared in the literature in a plethora of different contexts. Among the most prominent ones are so-called \textit{pruned Hurwitz numbers} \cite{do2018pruned,hahn2020bi}, \textit{monotone Hurwitz numbers} \cite{goulden2014monotone}, \textit{strictly monotone Hurwitz numbers} \cite{kazarian2015virasoro}, \textit{completed cycles Hurwitz numbers} \cite{okounkov2006gromov} and many more. For all of these variants the piecewise polynomiality of the double Hurwitz numbers analogue was established and for the majority a wall-crossing structure as well (see e.g. \cite{zbMATH06791415, hahn2018wall,hahn2020wall,shadrin2012double}).\\
While classical Hurwitz theory deals with the enumeration of branched morphisms between \textit{orientable surfaces}, it is very natural to ask for an analogous theory for non-orientable surfaces. Such a new and exciting theory for so-called \textit{$b$-Hurwitz numbers} was introduced in \cite{chapuy2020non} .

\subsection{Twisted Hurwitz numbers.}\label{sec-twistedsingle}
The construction of $b$-Hurwitz numbers is based on the following idea:  Let $\overline{\mathbb{H}}$ be the compactified complex upper halfplane of $\mathbb{P}^1$ and $\mathcal{J}$ the corresponding natural involution on $\mathbb{P}^1$. A generalised branched covering is a covering $f\colon S\to\overline{\mathbb{H}}$ where $S$ is a not-necessarily connected compact orientable surface with orientation double cover $\hat{S}$, such that $f$ may be "lifted" to a branched covering $\hat{S}\to\mathbb{P}^1$. We give a precise formulation in \cref{sec:twistdoub}. Via these generalised coverings the authors of \cite{chapuy2020non} introduce a new one-parameter deformation of classical Hurwitz numbers called $b$-Hurwitz numbers in reference to the $b$-conjecture by Goulden and Jackson in the context of Jack polynomials \cite{goulden1996connection}. In order to obtain $b$-Hurwitz numbers, the authors of \cite{chapuy2020non} associate a non-negative integer $\nu_p(f)$ to any generalised branched covering $f$ which "measures" the non-orientability of the the surface $\tilde{S}$. This non-negative integer is zero if and only if $\tilde{S}$ is orientable. Based on this idea $b$-Hurwitz numbers are defined -- depending on a \textit{measure of non-orientability} $p$ -- as a sum over generalised branched coverings weighted by $b^{\nu_p(f)}$. Thus, one obtains a Hurwitz-type enumerations for any value of $b$. For example, under the convention that $0^0=1$, one  recovers classical Hurwitz numbers for $b=0$. For $b=1$ one obtains enumerations of generalised branched coverings which are called \textit{twisted Hurwitz numbers}. This is the case we study in the present paper.\\
The term twisted Hurwitz numbers was coined in \cite{BF21} in the context of \textit{surgery theory}. Surgery theory studies the construction of new manifolds from given ones via cutting and glueing, such that key properties are preserved. In \cite{BF21}, the enumeration of decompositions of a given surface with boundary and marked points is studied. The term \textit{twisted} is motivated by the fact in \cite{BF21} gluings are performed with respect to a twist of the natural boundary orientations. It was proved in \cite{BF21} that the enumeration of certain decompositions with respect to such a twist may be computed in terms of factorisations in the symmetric group, reminiscent of Hurwitz' result in his original work \cite{hurwitz1892algebraische}. More precisely, we fix the involution $$\tau= (1\;\; n+1) (2 \;\;n+2) \ldots (n \;\;2n)\in \mathbb{S}_{2n}$$ and use the notation

\begin{equation}
    B_n=C(\tau)=\{\sigma \in \mathbb{S}_{2n}\;|\; \sigma \tau \sigma^{-1}=\tau\},\quad C^\sim(\tau)=\{\sigma \in \mathbb{S}_{2n}\;|\; \tau \sigma \tau^{-1} = \tau \sigma \tau =  \sigma^{-1}\}
\end{equation}

where $B_n$ is the hyperoctahedral group. We further define the subset $B^\sim_n \subset C^\sim(\tau)$ consisting of those permutations that have no self-symmetric cycles (see \cite[Lemma 2.1]{BF21}). We then set, for a partition $\lambda$ of $n$,  $B^\sim_\lambda\subset B^\sim_n$ to consist of those permutations that have $2\ell(\lambda)$ cycles, two of length $\lambda_i$ for each i, that pair up under conjugation with $\tau$. We are now ready to define \textbf{twisted single Hurwitz numbers} in terms of the symmetric group.

\begin{definition}[Twisted single Hurwitz numbers, \cite{BF21}]
\label{def:twisthur}
Fix a partition $\lambda$ of $n$ and a number $b$ (the number of transpositions).
Then define
\begin{align}\tilde{h}_{b}(\lambda)= \frac{1}{n!}\sharp \Big\{(\sigma_1,\ldots,\sigma_b)\;|\; \sigma_s=(i_s\;\;j_s),\, j_s\neq \tau(i_s),\, \sigma_1\ldots\sigma_b (\tau\sigma_b\tau)\ldots (\tau\sigma_1\tau) \in B^{\sim}_\lambda\Big\}.\end{align}
\end{definition}

Maybe surprisingly, it was then proved in \cite[Theorem 3.2]{BF21} that these numbers coincide with $b$-Hurwitz numbers for $b=1$ by showing that the generating series' of both invariants satisfy the same PDE with equal initial data.

\subsection{Tropical geometry of twisted Hurwitz numbers.} 
The present paper develops a tropical theory of twisted Hurwitz numbers and demonstrates some first applications. In \cref{sec:twistdoub}, we define a generalisation \cref{def:twisthur} to twisted double Hurwitz numbers $\tilde{h}_g(\mu,\nu)$ which by the same arguments as in \cite[Theorem 3.2]{BF21} coincides with $b$-Hurwitz numbers for $b=1$. This generalisation arises naturally from the symmetric group expression for twisted single Hurwitz numbers. Moreover, we define in \cref{sec:troptwist} a tropical analogue of twisted double Hurwitz numbers in terms of tropical covers. We prove in \cref{sec:corr} that twisted double Hurwitz numbers coincide with their tropical counterpart, thus giving a tropical correspondence theorem for these enumerative invariants. This allows us to derive a purely graph-theoretic interpretation of twisted double Hurwitz numbers in \cref{sec:monodr} by reinterpreting the tropical covers as directed graphs.
Finally, we employ this expression of twisted double Hurwitz numbers as a weighted enumeration of directed graphs to study the polynomiality of twisted Hurwitz numbers in \cref{sec:poly} improving \cite[Theorem 6.6]{chapuy2020non} in the case $b=1$. Finally, we discuss the wall-crossing behaviour of twisted Hurwitz numbers.

\subsection*{Acknowledgements.} 
We would like to thank an anonymous referee for their thorough work and for numerous helpful suggestions on how to improve the first versions of this paper. We thank Rapha\"el Fesler and Veronika K\"orber for useful discussions and comments.
The second author acknowledges support by the Deutsche Forschungsgemeinschaft (DFG, German Research Foundation), Project-ID 286237555, TRR 195. Computations have been made using the Computer Algebra System \textsc{gap} and \textsc{OSCAR} \cite{GAP4, Oscar}.

\section{Twisted double Hurwitz numbers}
\label{sec:twistdoub}
In this section, we define twisted double Hurwitz numbers as a factorization problem in the symmetric group, generalizing the case of single twisted Hurwitz numbers discussed in Subsection \ref{sec-twistedsingle}. We use the notation of Subsection \ref{sec-twistedsingle}. To begin with, we recall that

\begin{equation}
     C^\sim(\tau)=\{\sigma\in \mathbb{S}_{2n} \;|\; \tau \sigma \tau^{-1} = \tau \sigma \tau =  \sigma^{-1}\}.
\end{equation}

It was proved in \cite[Lemma 2.1]{BF21} for $\sigma \in C^\sim(\tau)$ with the decomposition in cycles $\sigma=c_1\cdots c_m$, we have for any $i$ that either
\begin{itemize}
    \item there exists $j\neq i$ with $\tau c_i\tau=c_j^{-1}$ or
    \item we have $\tau c_i\tau=c_i^{-1}$ and $c_i$ has even length.
\end{itemize}
In the first case, $c_i$ and $c_j$ are called \textit{$\tau$-symmetric}, while in the second case $c_i$ is called \textit{self-symmetric}. As mentioned above, we denote by $B_n^\sim\subset C^\sim(\tau)$ the set of permutations without self-symmetric cycles and by $B^\sim_\lambda\subset B^\sim_n$ the set of permutations in $B^\sim_n$ with $2\ell(\lambda)$ cycles and for each $i$ two cycles of length $\lambda_i$ that are $\tau$-symmetric.

\begin{definition}[Cycle type]
Let $\sigma\in B^{\sim}_n$. We denote its cycle type by $C(\sigma)$, which is a partition of $2n$ recording the lengths of the cycles of $\sigma$.
\end{definition}

For a partition $\mu$ of $n$ we denote by $2\mu$ the partition of $2n$ with twice as many parts, where each part is repeated once.

We may now define twisted double Hurwitz numbers generalising \cref{def:twisthur}.

\begin{definition}[Twisted double Hurwitz numbers]
\label{def:twisthurdou}
Let $g\ge0,n>0$ and $\mu,\nu$ partitions of $n$. We define $C_g(\mu,\nu)$ as the set of tuples $(\sigma_1,\eta_1,\dots,\eta_b,\sigma_2)$,
such that we have:
\begin{enumerate}
    \item $b=\frac{2g-2+2\ell(\mu)+2\ell(\nu)}{2}>0$,
    \item $C(\sigma_1)=2\mu$, $C(\sigma_2)=2\nu$, $\eta_i$ are transpositions satisfying $\eta_i\neq \tau \eta_i \tau$,
    \item $\sigma_1\in B^\sim_\mu$,
    \item $\eta_b\cdots\eta_1\sigma_1(\tau\eta_1\tau)\cdots(\tau\eta_b\tau)=\sigma_2$,
    \item the subgroup 
    $$\langle \sigma_1,\eta_1,\ldots,\eta_{b},\tau \eta_1\tau, \ldots,\tau\eta_b\tau,\sigma_2\rangle$$ acts transitively on the set $\{1,\ldots,2d\}$
\end{enumerate}
Then, we define the associated \textit{twisted double Hurwitz number} as
\begin{equation}
    \tilde{h}_{g}(\mu,\nu)=\frac{1}{(2n)!!}|C_g(\mu,\nu)|.
\end{equation}
When we drop the transitivity condition, we obtain possibly disconnected twisted double Hurwitz numbers which we denote $\tilde{h}_{g}^\bullet(\mu,\nu)$.
\end{definition}

\begin{remark}[Conventional differences]
Compared with the definition of twisted single Hurwitz numbers in Definition \ref{def:twisthur}, there are two conventional differences:
\begin{enumerate}
    \item Rather than the number of transpositions, we use the genus (of the source of a twisted tropical cover, see Definition \ref{def-twistedtropcover}) as subscript in the notation. In the usual case of Hurwitz numbers (without twisting) counting covers, this corresponds to the genus of the source curves. By the Riemann-Hurwitz formula, the genus $g$ and the number $b$ of transpositions (i.e.\ simple branch points) are related via
    $$b=\frac{2g-2+2\ell(\mu)+2\ell(\nu)}{2}.$$
    
    \item We choose to normalize with the factor $\frac{1}{(2n)!!}$ rather than $\frac{1}{n!}$. This leads to nicer formulae and structural results.
\end{enumerate}

\end{remark}
\begin{remark}[Connectedness and transitivity]
    As noted above, one can also drop the transitivity condition (5) in Definition \ref{def:twisthurdou} to obtain $h_g^\bullet(\mu,\nu)$. On the tropical side, this amounts to allowing disconnected tropical curves as source of a twisted tropical cover, see Definition \ref{def-twistedtropcover}. In this setting, it can happen that a disconnected twisted tropical cover contains two twisted components which both just correspond to a single edge without any interior vertex. This also happens in the connected case if we allow $b=0$. For this case, one has to adapt the tropical multiplicity we set in Definition \ref{def:twisttrophn}: a twisted tropical cover which consists of a pair of twisted single edges of weight $\mu$ each has multiplicity $\frac{1}{\mu}$. With this adaption, one can easily generalize our results to the disconnected case.
\end{remark}

As mentioned in the introduction, twisted Hurwitz numbers first appeared in the Hurwitz theory of non-orientable surfaces. More precisely, we denote by $\mathcal{J}\colon\mathbb{P}^1\to\mathbb{P}^1$ the complex conjugation, by $\mathbb{H}\coloneqq\{z\in\mathbb{C}\mid\mathrm{Im}(z)\ge0\}$ the complex upper halfplane and let $\overline{\mathbb{H}}\coloneqq\mathbb{H}\cup\{\infty\}$ the compactified complex upper halfplane. Moreover, we denote by $\pi\colon\mathbb{P}^1\to\overline{\mathbb{H}}$ the quotient map.\\
We call continuous maps $f\colon S\to\overline{\mathbb{H}}$ \textit{generalised branched coverings}, if $S$ is a not-necessarily orientable surface and there exists a further map $\hat{f}\colon \hat{S}\to\mathbb{P}^1$ with
    \begin{enumerate}
        \item $p\colon\hat{S}\to S$ is the orientation double cover,
        \item $\pi\circ\hat{f}=f\circ p$,
        \item all the branch points of $\hat{f}$ are real.
    \end{enumerate}
Let $\mathcal{T}\colon \hat{S}\to\hat{S}$ be an orientation reversing involution without fix points, such that $p\circ \mathcal{T}=p$. Then, the second condition of $\hat{f}$ may be reformulated as $\hat{f}\circ\mathcal{T}=\mathcal{J}\circ \hat{f}$. As $\mathcal{T}$ has no fixed points, for any branch points $c\in\mathbb{P}^1_{\mathbb{R}}\subset\mathbb{P}^1$ of $\hat{f}$, the points in its pre-image come in pairs $(a,\mathcal{T}(a))$ with the same ramification index. Thus, the degree of $\hat{f}$ is even and the ramification profile of $c$ repeats any entry twice, e.g. $(\lambda_1,\lambda_1,\dots,\lambda_s,\lambda_s)$. We then say that $f$ has ramification profile $(\lambda_1,\dots,\lambda_s)$ at $\pi(s)\in\partial\overline{\mathbb{H}}$. We further call two generalised branched coverings $f_1$ and $f_2$ equivalent if their lifts $\hat{f_1}$ and $\hat{f_2}$ are, and denote the equivalence class of $f$ by $[f]$. We are now ready to define

\begin{definition}\label{def-nonoriented}
    Let $g\ge0$, $n>0$, $\mu,\nu$ partitions of $n$. Let $b=\frac{2g-2+2\ell(\mu)+2\ell(\nu)}{2}$ and fix $p_1,\dots,p_b$ pairwise distinct real points on $\overline{\mathbb{H}}$. We define $G_g(\mu,\nu)$ as the set of equivalence classes $[f]$ of generalised branched coverings $f\colon S\to\overline{\mathbb{H}}$, such that
    \begin{itemize}
        \item $f$ is of degree $n$,
        \item $f$ has ramification profile $\mu$ over $0$ and $\nu$ over $\infty$,
        \item $f$ has ramification profile $(2,1,\dots,1)$ over $p_i$.
    \end{itemize}
    Then, we define $1$-Hurwitz numbers as
    \begin{equation}
        h^1_g(\mu,\nu)=\sum_{[f]\in G_g(\mu,\nu)}\frac{1}{|\mathrm{Aut}(f)|}.
    \end{equation}
\end{definition}
The parameter $g$ used here is not equal to the genus of the surface $S$, this is $g'=\frac{g+1}{2}.$

The following was proved in \cite[Theorem 3.2]{BF21} for $\nu=(1,\dots,1)$. However, the same approach works for arbitrary $\nu$. In particular, the idea is that in \cite[Theorem 6.5]{chapuy2020non} a recursion of $b$--Hurwitz numbers was derived. Moreover, twisted Hurwitz numbers were proved to satisfy a recursion in \cite[Theorem 2.12]{BF21}. It turns out that these recursions only differ by a factor of $2$. The same argument as in the proof of \cite[Theorem 3.2]{BF21} proves the following result.

\begin{theorem}
    Let $g\ge0$, $n>0$ and $\mu$, $\nu$ partitions of $d$. Then, we have
    \begin{equation}
        h^1_g(\mu,\nu)=2^{-b}\tilde{h}_g^{\bullet}(\mu,\nu).
    \end{equation}
\end{theorem}

\section{Tropical twisted covers and twisted Hurwitz numbers}
\label{sec:troptwist}

In \cite{CJM10}, tropical Hurwitz numbers have been introduced as a count of tropical covers, parallel to the count of covers of algebraic curves from Definition \ref{def-hur}. We start by recalling the basic notions of tropical curves and covers. Then we introduce twisted tropical covers, which can roughly be viewed as tropical covers with an involution. By fixing branch points, we produce a finite count of twisted tropical covers for which we show in the following section that it coincides with the corresponding twisted double Hurwitz number. Readers with a background in the theory of tropical curves are pointed to the fact that we only consider explicit tropical curves in the following, i.e.\ there is no genus hidden at vertices.

\begin{definition}[Abstract tropical curves]
An abstract tropical curve is a connected graph $\Gamma$ with the following data:
\begin{enumerate}
    \item The vertex set of $\Gamma$ is denoted by $V(\Gamma)$ and the edge set of $\Gamma$ is denoted by $E(\Gamma)$.
    \item The $1$-valent vertices of $\Gamma$ are called \textit{leaves} and the edges adjacent to leaves are called \textit{ends}.
    \item The set of edges $E(\Gamma)$ is partitioned into the set of ends $E^\infty(\Gamma)$ and the set of \textit{internal edges} $E^0(\Gamma)$.
    \item There is a length function
    \begin{equation}
        \ell\colon E(\Gamma)\to\mathbb{R}\cup\{\infty\},
    \end{equation}
    such that $\ell^{-1}(\infty)=E^\infty(\Gamma)$.
\end{enumerate}
The genus of an abstract tropical curve $\Gamma$ is defined as the first Betti number of the underlying graph, i.e.\ $g=1-\# E(\Gamma)+\# V(\Gamma)$. An isomorphism of abstract tropical curves is an isomorphism of the underlying graphs that respects the length function. The combinatorial type of an abstract tropical curve is the underlying graph without the length function.
\end{definition}

We are now ready to define the notion of a tropical cover.

\begin{definition}[Tropical covers]
A tropical cover is a surjective harmonic map between abstract tropical curves $\pi\colon \Gamma_1\to\Gamma_2$, i.e.:
\begin{enumerate}
    \item We have $\pi(V(\Gamma_1))\subset V(\Gamma_2)$ and $\pi(E(\Gamma_1))\subset E(\Gamma_2)\cup V(\Gamma_2)$.
    \item Let $e\in E(\Gamma_1)$, such that $\pi(e)\in E(\Gamma_2)$ (otherwise $e$ is contracted to a vertex under $\pi$). Then, we interpret $e$ and $\pi(e)$ as intervals $[0,\ell(e)]$ and $[0,\ell(\pi(e))]$ respectively. We require $\pi$ restricted to $e$ to be a bijective integer linear function $[0,\ell(e)]\to[0,\ell(\pi(e))]$ given by $t\mapsto \omega(e)\cdot t$, with $\omega(e) \in \mathbb{Z}$. If $\pi(e)\in V(\Gamma_2)$, we define $\omega(e)=0$. We call $\omega(e)$ the \textit{weight} of $e$.
    \item For a vertex $v\in V(\Gamma_1)$, we set $v'=\pi(v)\in V(\Gamma_2)$. We choose an edge $e'$ adjacent to $v'$. Then, we define the \textit{local degree} $d_v$ at $v$ as
    \begin{equation}
        d_v=\sum_{\substack{e\in \Gamma_1,\; v \in \partial e\\e^0\cap\pi^{-1}((e')^0)\neq \emptyset}}\omega(e).
    \end{equation}
    Here, we denote by $e^0$ the interior of the edge $e$.
    We require $d_v$ to be independent of the choice of $e'$. We call this independency the \textit{harmonicity} or \textit{balancing condition}.
\end{enumerate}
Moreover, we define the \textit{degree} of $\pi$ as the sum of local degrees at all vertices of $\Gamma_1$ in the pre-image of a given vertex of $\Gamma_2$. The degree is independent of the choice of vertex of $\Gamma_2$. This follows from the harmonicity condition.\\
For any end $e$ of $\Gamma_2$, we define a partition $\mu_e$ as the partition of weights of ends of $\Gamma_1$ mapping to $e$. We call $\mu_e$ the \textit{ramification profile} of $e$.\\
We call two tropical covers $\pi_1\colon\Gamma_1\to\Gamma_2$ and $\pi_2\colon\Gamma_1'\to\Gamma_2$ equivalent if there exists an isomorphism $g\colon\Gamma_1\to\Gamma_1'$ of metric graphs, such that $\pi_2\circ g=\pi_1$.
\end{definition}

\begin{definition}[Twisted tropical covers]\label{def-twistedtropcover}
We define a twisted tropical cover of type $(g,\mu,\nu)$ to be a tropical cover $\pi\colon\Gamma_1\to\Gamma_2$ with an involution $\iota\colon\Gamma_1\to\Gamma_1$ which respects the cover $\pi$, such that:
\begin{itemize}
    \item The target $\Gamma_2$ is a subdivided version of $\mathbb{R}$ with vertices $p_1,\dots,p_b=V(\Gamma_2)$, where $p_i<p_{i+1}$. Here, $b=\frac{2g-2+2\ell(\mu)+2\ell(\nu)}{2}$. These points are called the \emph{branch points}.
    \item There are $\ell(\mu)$ pairs of ends mapping to $(-\infty,p_1]$ with weights $\mu_1,\dots,\mu_{\ell(\mu)}$ and $\ell(\nu)$ pairs of ends mapping to $[p_b,\infty)$ with weights $\nu_1,\ldots,\nu_{\ell(\nu)}$.
    \item in the preimage of each point $p_i$, there are either two $3$-valent vertices or one $4$-valent vertex.
    \item the edges adjacent to a $4$-valent vertex all have the same weight.
    \item the fixed locus of $\iota$ is exactly the set of $4$-valent vertices.
\end{itemize}
\end{definition}

\begin{definition}[Quotient graph $\Gamma/\iota$]
Let $\pi\colon \Gamma\to \mathbb{R}$ be a twisted tropical cover with involution $\iota\colon\Gamma\to\Gamma$. The involution $\iota$ induces a symmetric relation on the vertex and edge sets of $\Gamma$: We define for $v,v'\in V(\Gamma)$  (resp. $e,e'\in E(\Gamma)$) that $v\sim v'$ (resp. $e\sim_\iota e'$) if and only $\iota(v)=v'$ (resp. $\iota(e)=e'$). We define $\Gamma/\iota$ as the graph with vertex set $V(\Gamma)/\sim$ and edge set $E(\Gamma)/\sim$ with natural identifications. For $e=[e',e'']\in E(\Gamma/\iota)$ we define the length $\ell(e)$ as $\ell(e')=\ell(e'')$ and its weight $\omega(e)$ with respect to $\pi$ to be the weight $\omega(e')=\omega(e'')$.
\end{definition}

\begin{example}
The graph on the left in \cref{fig:extwist} illustrates a twisted tropical cover $\pi$ of type $(0,(4),(2,2))$. We fix $p_1=-5$ and $p_2=5$. The graph on the top represents the source curve $\Gamma_1$, whereas the graph $\Gamma_2$ on the bottom is $\mathbb{R}$ with two marked points $p_1$ and $p_2$. The graph above is mapped to the graph below via the projection indicated by the arrow, e.g. the two left most edges of $\Gamma_1$ are each mapped to $(-\infty,-5)$. The green labels indicate the lengths of the corresponding edge. We have six unbounded edges of $\Gamma_1$, whose lengths we denote by $\infty$ and two bounded edges, each of length five and each mapped to $(-5,5)$. The graph $\Gamma_2$ has two unbounded edges corresponding to the intervals $(-\infty,-5)$ and $(5,\infty)$ and one bounded edge which corresponds to $(-5,5)$. The red labels indicate the weights $\omega(e)$ of the corresponding edges $e$. For the two bounded edges of $\Gamma_1$, we see immediately that the corresponding weight has to be $2$, since $\ell(\pi(e))=\omega(e)\cdot \ell(e)$. For the unbounded edges the weight cannot be read of the difference of lengths of the original edge and its image. However, the choices of weights in \cref{fig:extwist} satisfy the balancing condition. The involution $\iota$ can be visualized in the picture as the reflection along a horizontal line passing through the $4$-valent vertex. We can see that the picture
yields a twisted tropical cover of type $(0,(4),(2,2))$.\\
On the right of \cref{fig:extwist} the quotient graph $\Gamma/\iota$ is illustrated.
\end{example}

\begin{definition}[Automorphisms]
    Let $\pi:\Gamma\rightarrow \mathbb{R}$ be a twisted tropical cover with involution $\iota:\Gamma\rightarrow \Gamma$. An automorphism of $\pi$ is a morphism of abstract tropical curves (i.e.\ a map of metric graphs) $f:\Gamma\rightarrow\Gamma$ respecting the cover and the involution, i.e.\ $\pi\circ f = \pi$ and $f\circ \iota = \iota \circ f$. We denote the group of automorphisms of $\pi$ by $\Aut(\pi)$.
\end{definition}

\begin{figure}
\tikzset{every picture/.style={line width=0.75pt}} 

\begin{tikzpicture}[x=0.75pt,y=0.75pt,yscale=-1,xscale=1,scale=0.93]

\draw [line width=1.5]    (13.74,63.25) -- (81.59,63.25) ;
\draw [line width=1.5]    (81.59,63.25) .. controls (109.62,27.48) and (239.42,16.99) .. (286.62,16.99) ;
\draw [line width=1.5]    (81.59,63.25) .. controls (109.62,100.95) and (190.75,75.3) .. (198.86,133.6) ;
\draw [line width=1.5]    (13,203.18) -- (80.85,203.18) ;
\draw [line width=1.5]    (80.85,203.18) .. controls (108.88,250.22) and (245.32,250.22) .. (286.62,249.05) ;
\draw [line width=1.5]    (80.85,203.18) .. controls (109.62,168.59) and (190.75,190.74) .. (198.86,133.6) ;
\draw [line width=1.5]    (198.86,133.6) .. controls (226.88,98.62) and (242.37,98.62) .. (287.36,98.62) ;
\draw [line width=1.5]    (198.86,133.6) .. controls (228.36,168.59) and (243.11,167.42) .. (287.36,168.59) ;
\draw [line width=1.5]    (149.69,263.09) -- (149.69,320.05) ;
\draw [shift={(149.69,323.05)}, rotate = 270] [color={rgb, 255:red, 0; green, 0; blue, 0 }  ][line width=1.5]    (14.21,-4.28) .. controls (9.04,-1.82) and (4.3,-0.39) .. (0,0) .. controls (4.3,0.39) and (9.04,1.82) .. (14.21,4.28)   ;
\draw [line width=1.5]    (13.97,374.27) -- (286.88,375.52) ;
\draw [line width=1.5]    (81.59,386.77) -- (81.59,363.03) ;
\draw [line width=1.5]    (198.34,385.52) -- (198.34,361.78) ;
\draw [color={rgb, 255:red, 155; green, 155; blue, 155 }  ,draw opacity=1 ][line width=1.5]  [dash pattern={on 5.63pt off 4.5pt}]  (13,132) -- (287.36,132) ;
\draw    (290.57,63.44) .. controls (327.96,103.22) and (329.28,170.72) .. (290.02,209.26) ;
\draw [shift={(288.2,211)}, rotate = 317.09] [fill={rgb, 255:red, 0; green, 0; blue, 0 }  ][line width=0.08]  [draw opacity=0] (8.93,-4.29) -- (0,0) -- (8.93,4.29) -- cycle    ;
\draw [shift={(288.2,61)}, rotate = 44.8] [fill={rgb, 255:red, 0; green, 0; blue, 0 }  ][line width=0.08]  [draw opacity=0] (8.93,-4.29) -- (0,0) -- (8.93,4.29) -- cycle    ;
\draw [line width=1.5]    (342.74,132.74) -- (410.59,132.74) ;
\draw [line width=1.5]    (410.59,132.74) .. controls (438.62,96.98) and (568.42,86.48) .. (615.62,86.48) ;
\draw [line width=1.5]    (410.59,132.74) .. controls (438.62,170.44) and (519.75,144.79) .. (527.86,203.1) ;
\draw [line width=1.5]    (527.86,203.1) .. controls (555.88,168.11) and (571.37,168.11) .. (616.36,168.11) ;

\draw (73.28,392.91) node [anchor=north west][inner sep=0.75pt]    {$-5$};
\draw (192.34,392.91) node [anchor=north west][inner sep=0.75pt]    {$5$};
\draw (31.98,39.37) node [anchor=north west][inner sep=0.75pt]  [color={rgb, 255:red, 65; green, 117; blue, 5 }  ,opacity=1 ]  {$\infty $};
\draw (31.5,180.54) node [anchor=north west][inner sep=0.75pt]  [color={rgb, 255:red, 65; green, 117; blue, 5 }  ,opacity=1 ]  {$\infty $};
\draw (28.58,351.68) node [anchor=north west][inner sep=0.75pt]  [color={rgb, 255:red, 65; green, 117; blue, 5 }  ,opacity=1 ]  {$\infty $};
\draw (248.46,355.43) node [anchor=north west][inner sep=0.75pt]  [color={rgb, 255:red, 65; green, 117; blue, 5 }  ,opacity=1 ]  {$\infty $};
\draw (231.43,228.01) node [anchor=north west][inner sep=0.75pt]  [color={rgb, 255:red, 65; green, 117; blue, 5 }  ,opacity=1 ]  {$\infty $};
\draw (250.4,146.81) node [anchor=north west][inner sep=0.75pt]  [color={rgb, 255:red, 65; green, 117; blue, 5 }  ,opacity=1 ]  {$\infty $};
\draw (247.97,79.35) node [anchor=north west][inner sep=0.75pt]  [color={rgb, 255:red, 65; green, 117; blue, 5 }  ,opacity=1 ]  {$\infty $};
\draw (212.95,1.89) node [anchor=north west][inner sep=0.75pt]  [color={rgb, 255:red, 65; green, 117; blue, 5 }  ,opacity=1 ]  {$\infty $};
\draw (133.97,66.85) node [anchor=north west][inner sep=0.75pt]  [color={rgb, 255:red, 65; green, 117; blue, 5 }  ,opacity=1 ]  {$5$};
\draw (130.56,156.8) node [anchor=north west][inner sep=0.75pt]  [color={rgb, 255:red, 65; green, 117; blue, 5 }  ,opacity=1 ]  {$5$};
\draw (132.17,347.93) node [anchor=north west][inner sep=0.75pt]    {$\textcolor[rgb]{0.25,0.46,0.02}{10}$};
\draw (33.76,75.6) node [anchor=north west][inner sep=0.75pt]    {$\textcolor[rgb]{0.82,0.01,0.11}{4}$};
\draw (34.24,215.51) node [anchor=north west][inner sep=0.75pt]    {$\textcolor[rgb]{0.82,0.01,0.11}{4}$};
\draw (128.13,96.84) node [anchor=north west][inner sep=0.75pt]    {$\textcolor[rgb]{0.82,0.01,0.11}{2}$};
\draw (133.48,186.78) node [anchor=north west][inner sep=0.75pt]    {$\textcolor[rgb]{0.82,0.01,0.11}{2}$};
\draw (232.72,259.24) node [anchor=north west][inner sep=0.75pt]    {$\textcolor[rgb]{0.82,0.01,0.11}{2}$};
\draw (253.15,176.79) node [anchor=north west][inner sep=0.75pt]    {$\textcolor[rgb]{0.82,0.01,0.11}{2}$};
\draw (251.2,108.08) node [anchor=north west][inner sep=0.75pt]    {$\textcolor[rgb]{0.82,0.01,0.11}{2}$};
\draw (215.69,33.12) node [anchor=north west][inner sep=0.75pt]    {$\textcolor[rgb]{0.82,0.01,0.11}{2}$};
\draw (321.04,120) node [anchor=north west][inner sep=0.75pt]   [align=left] {$\displaystyle \iota $};
\draw (360.98,108.86) node [anchor=north west][inner sep=0.75pt]  [color={rgb, 255:red, 65; green, 117; blue, 5 }  ,opacity=1 ]  {$\infty $};
\draw (576.97,148.84) node [anchor=north west][inner sep=0.75pt]  [color={rgb, 255:red, 65; green, 117; blue, 5 }  ,opacity=1 ]  {$\infty $};
\draw (541.95,71.38) node [anchor=north west][inner sep=0.75pt]  [color={rgb, 255:red, 65; green, 117; blue, 5 }  ,opacity=1 ]  {$\infty $};
\draw (462.97,136.35) node [anchor=north west][inner sep=0.75pt]  [color={rgb, 255:red, 65; green, 117; blue, 5 }  ,opacity=1 ]  {$5$};
\draw (362.76,145.09) node [anchor=north west][inner sep=0.75pt]    {$\textcolor[rgb]{0.82,0.01,0.11}{4}$};
\draw (457.13,166.33) node [anchor=north west][inner sep=0.75pt]    {$\textcolor[rgb]{0.82,0.01,0.11}{2}$};
\draw (580.2,177.57) node [anchor=north west][inner sep=0.75pt]    {$\textcolor[rgb]{0.82,0.01,0.11}{2}$};
\draw (544.69,102.62) node [anchor=north west][inner sep=0.75pt]    {$\textcolor[rgb]{0.82,0.01,0.11}{2}$};

\end{tikzpicture}
\caption{}
\label{fig:extwist}
\end{figure}

\begin{definition}[Twisted tropical Hurwitz number]
\label{def:twisttrophn}
We define the tropical twisted double Hurwitz 
number $\tilde{h}^{\trop}_g(\mu,\nu)$ to be  the weighted enumeration of equivalence classes of twisted tropical covers of type $(g,\mu,\nu)$, such that each equivalence class $[\pi\colon \Gamma\to\mathbb{R}]$ is counted with multiplicity
\begin{equation}
    2^{b}\cdot \frac{1}{|\Aut(\pi)|}\cdot\prod_V (\omega_V-1)\prod_e\omega(e),
\end{equation}
where $b$ is the number of branch points. Moreover, the first product goes over all $4$-valent vertices and $\omega_V$ denotes the weight of the adjacent edges, while the second product is taken over all internal edges of the quotient graph $\Gamma/\iota$ and $\omega(e)$ denotes their weights.

\end{definition}

\section{The Correspondence Theorem}
\label{sec:corr}


In this section, we show that  twisted double Hurwitz numbers coincide with their tropical counterparts.
In Construction \ref{con-twist}, we associate a twisted tropical cover to a tuple in the symmetric group counting towards a twisted Hurwitz number. In Lemma \ref{lem-outcomeconstruction}, we show that the outcome is indeed a twisted tropical cover. In Proposition \ref{thm-covermult}, we show that the multiplicity with which a twisted tropical cover is counted exactly equals the number of tuples that are mapped to it via Construction \ref{con-twist}. This proof builds on \cite[Theorem 2.12]{BF21}, where the cut-and-join operator for twisted Hurwitz numbers was derived.

The result is summed up in Theorem \ref{thm-corres}.

\begin{construction}
\label{con-twist}
Let $(\sigma_1,
\eta_1,\dots,\eta_b,\sigma_2)\in C_g(\mu,\nu)$. We let $\Gamma_2$ be $\mathbb{R}$ subdivided by $b$ vertices $p_i$ and construct a twisted tropical cover $\pi\colon\Gamma_1\to\Gamma_2$ associated to this tuple.

\begin{enumerate}
    \item First, we fix $p_1,\dots,p_b\in V(\Gamma_2)$ with $p_i<p_{i+1}$ and set $p_0=-\infty$ and $p_{b+1}=
    \infty$.
    \item We begin with $2\ell(\mu)$ ends over the interval $(-\infty,p_1]$, labelled by $\sigma_1^1,\dots,\sigma_1^{2\ell(\mu)}$, where $\sigma_1=\sigma_1^1\circ \ldots \circ \sigma_1^{2\ell(\mu)}$ is a decomposition of $\sigma_1$ into cycles. The action of $\tau$ on $\sigma_1$ yields an involution on these ends.
    \item By \cite[Theorem 2.14]{BF21}, the product
    \begin{equation}
    \label{equ:firststep}
        \eta_1\circ\sigma_1\circ(\tau\eta_1\tau)
    \end{equation}
    gives rise to three cases.
    \begin{enumerate}
        \item In this first case, we have four cycles $\sigma_1^1,\sigma_1^2,\sigma_1^3,\sigma_1^4$ of $\sigma_1$, such that
        \begin{equation}
            \sigma_1^1\sigma_1^2\sigma_1^3\sigma_1^4\in B^{\sim}_n.
        \end{equation}
        Then, the product in \cref{equ:firststep} joins two pairs of these cycles to two new cycles, e.g. $\sigma_1^1$, $\sigma_1^2$ to a new cycle $\eta_1\sigma_1^1\sigma_1^2\tau\eta_1\tau$; and $\sigma_1^3$, $\sigma_1^4$ to a new cycle $\eta_1\sigma_1^3\sigma_1^4\tau\eta_1\tau$.
        In this case, we create two vertices over $p_1$ each adjacent to two ends. We attach the ends that are joint via \cref{equ:firststep} to the same vertex. Moreover, we attach to each vertex two edges projecting to $[p_1,p_2]$ that are temporarily labeled by the corresponding cycles obtained from the join. Later, we replace this label by the length of the cycle as weight. This is illustrated in the following picture for the case that $\sigma_1^1$ is joined with $\sigma_1^2$ and $\sigma_1^3$ is joined with $\sigma_1^4$. Again, the action of $\tau$ on the permutation obtained from the join yields an involution of the tropical cover.

        
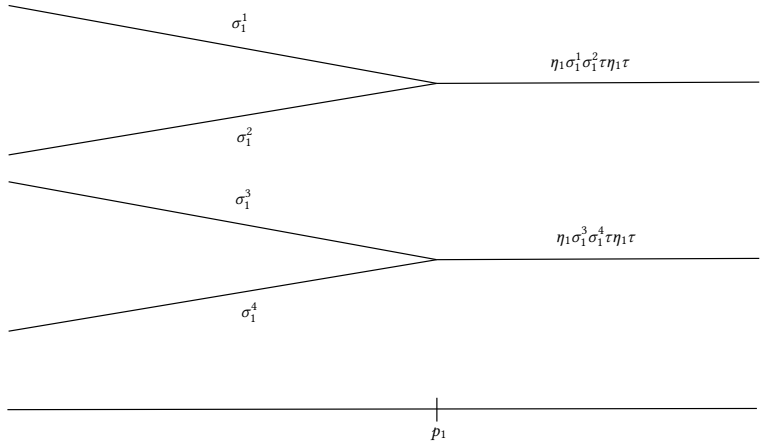
\begin{figure}[H]

        \scalebox{0.6}{
        \tikzset{every picture/.style={line width=0.75pt}} 
        
        \begin{tikzpicture}[x=0.75pt,y=0.75pt,yscale=-1,xscale=1]
        
        \draw    (51,30.5) -- (410.89,96.05) ;
        \draw    (51,156.22) -- (410.89,96.05) ;
        \draw    (410.89,96.05) -- (682,94.97) ;
        \draw    (51,178.78) -- (410.89,244.33) ;
        \draw    (51,304.5) -- (410.89,244.33) ;
        \draw    (410.89,244.33) -- (682,243.25) ;
        \draw    (50,370.5) -- (680,369.5) ;
        \draw    (411,360.5) -- (411,380.5) ;
        
        \draw (237.04,35.08) node [anchor=north west][inner sep=0.75pt]   [align=left] {$\displaystyle \sigma _{1}^{1}$};
        \draw (241.03,131.17) node [anchor=north west][inner sep=0.75pt]   [align=left] {$\displaystyle \sigma _{1}^{2}$};
        \draw (240.04,183.36) node [anchor=north west][inner sep=0.75pt]   [align=left] {$\displaystyle \sigma _{1}^{3}$};
        \draw (245.03,278.46) node [anchor=north west][inner sep=0.75pt]   [align=left] {$\displaystyle \sigma _{1}^{4}$};
        \draw (405,384) node [anchor=north west][inner sep=0.75pt]   [align=left] {$\displaystyle p_{1}$};
        \draw (505,68) node [anchor=north west][inner sep=0.75pt]   [align=left] {$\displaystyle \eta_1\sigma_1^1\sigma_1^2\tau\eta_1\tau$};
        \draw (510,215) node [anchor=north west][inner sep=0.75pt]   [align=left] {$\displaystyle \eta_1\sigma_1^3\sigma_1^4\tau\eta_1\tau$};

        \end{tikzpicture}}
        
        \end{figure}
           
        \item In the second case, we have two cycles $\sigma_1^1,\sigma_1^2$ of $\sigma_1$, such that
        \begin{equation}
            \sigma_1^1\sigma_1^2\in B^{\sim}_n.
        \end{equation}
        Then the product in \cref{equ:firststep} splits two cycles $\sigma_1^1$ and $\sigma_1^2$ each into two cycles, that we denote by $\sigma_1^{1,a},\sigma_1^{1,b}$ and $\sigma_1^{2,a},\sigma_1^{2,b}$ respectively. In this case, we create two vertices over $p_1$ each adjacent to one end labeled by $\sigma_1^1$ and $\sigma_1^2$ respectively. Moreover, we attach to each vertex two edges projecting to $[p_1,p_2]$ that are temporarily labeled by the corresponding cycles obtained from the split, i.e.\ the new edges attached to the vertex adjacent to $\sigma_1^i$ are labeled by $\sigma_1^{i,a}$ and $\sigma_1^{i,b}$. We illustrate this construction in the following picture.
         
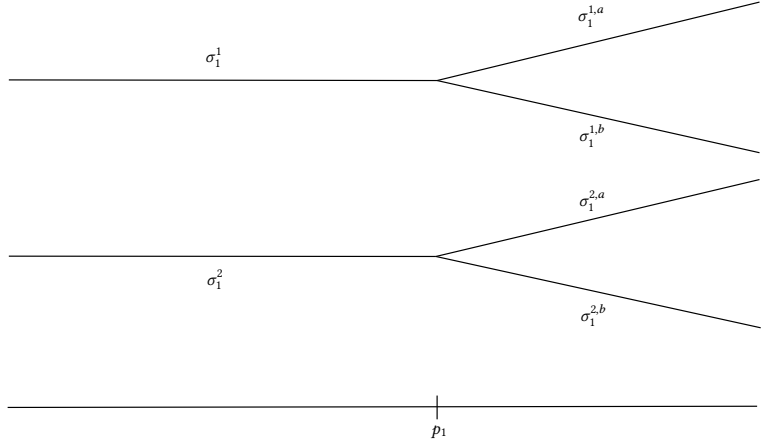
\begin{figure}[H]
        \scalebox{0.6}{

        \tikzset{every picture/.style={line width=0.75pt}} 
        
        \begin{tikzpicture}[x=0.75pt,y=0.75pt,yscale=-1,xscale=1]
        
        \draw    (682,29.5) -- (411,95.5) ;
        \draw    (682,156.22) -- (411,95.5) ;
        \draw    (411,95.5) -- (51,94.97) ;
        \draw    (682,178.78) -- (410,243.5) ;
        \draw    (683,303.5) -- (410,243.5) ;
        \draw    (410,243.5) -- (51,243.25) ;
        \draw    (50,370.5) -- (680,369.5) ;
        \draw    (411,360.5) -- (411,380.5) ;
        
        \draw (405,384) node [anchor=north west][inner sep=0.75pt]   [align=left] {$\displaystyle p_{1}$};
        \draw (215,66) node [anchor=north west][inner sep=0.75pt]   [align=left] {$\displaystyle \sigma _{1}^{1}$};
        \draw (216,253) node [anchor=north west][inner sep=0.75pt]   [align=left] {$\displaystyle \sigma _{1}^{2}$};
        \draw (528,31) node [anchor=north west][inner sep=0.75pt]   [align=left] {$\displaystyle \sigma _{1}^{1,a}$};
        \draw (529,131) node [anchor=north west][inner sep=0.75pt]   [align=left] {$\displaystyle \sigma _{1}^{1,b}$};
        \draw (529,186) node [anchor=north west][inner sep=0.75pt]   [align=left] {$\displaystyle \sigma _{1}^{2,a}$};
        \draw (530,282) node [anchor=north west][inner sep=0.75pt]   [align=left] {$\displaystyle \sigma _{1}^{2,b}$};

        \end{tikzpicture}

        }
        \end{figure}

        \item In the third case, we have two cycles $\sigma_1^1$ and $\sigma_1^2$ of $\sigma_1$ of the same length. Here, the product in \cref{equ:firststep} rearranges $\sigma_1^1$ and $\sigma_1^2$ into two new cycles $\tilde{\sigma}_1^1,\tilde{\sigma}_1^2$ of the same length. In this case, we create one vertex over $p_1$ that joins the two ends labeled by $\sigma_1^1,\sigma_1^2$. Moreover, we attach two new edges to this vertex that map to $[p_1,p_2]$ and that are labeled by $\tilde{\sigma}_1^1,\tilde{\sigma}_1^2$. 
        We illustrate this construction in the following picture.
    
        
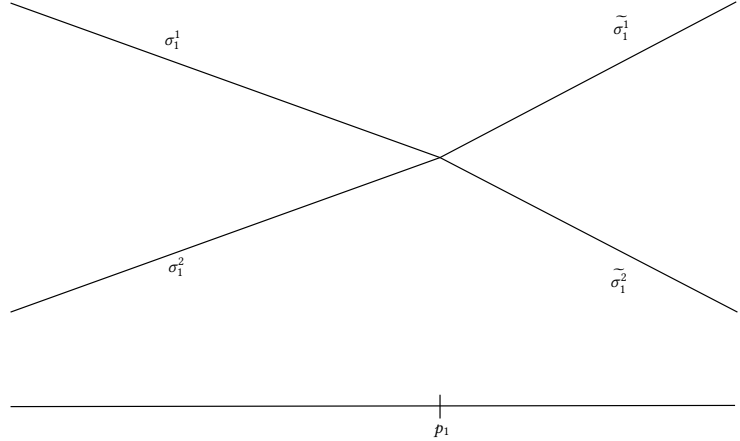
\begin{figure}[H]
        
        \tikzset{every picture/.style={line width=0.75pt}} 
        
        \scalebox{0.6}{\begin{tikzpicture}[x=0.75pt,y=0.75pt,yscale=-1,xscale=1]
        
        \draw    (50,370.5) -- (659,369.5) ;
        \draw    (411,360.5) -- (411,380.5) ;
        \draw    (50,31) -- (411,161) ;
        \draw    (50,291) -- (411,161) ;
        \draw    (660,30) -- (411,161) ;
        \draw    (661,291) -- (411,161) ;
        
        \draw (405,384) node [anchor=north west][inner sep=0.75pt]   [align=left] {$\displaystyle p_{1}$};
        \draw (178,52) node [anchor=north west][inner sep=0.75pt]   [align=left] {$\displaystyle \sigma _{1}^{1}$};
        \draw (553,253) node [anchor=north west][inner sep=0.75pt]   [align=left] {$\displaystyle \widetilde{\sigma _{1}^{2}}$};
        \draw (555,40) node [anchor=north west][inner sep=0.75pt]   [align=left] {$\displaystyle \widetilde{\sigma _{1}^{1}}$};
        \draw (181,243) node [anchor=north west][inner sep=0.75pt]   [align=left] {$\displaystyle \sigma _{1}^{2}$};
        
        \end{tikzpicture}}

        \end{figure}

    \end{enumerate}
    In each case, we further extend all ends not attached to a vertex over $p_1$ to $[p_1,p_2]$.
    \item We now take the permutation $\eta_1\sigma_1\tau\eta_1\tau$ and consider the product
    \begin{equation}
        \eta_2(\eta_1\sigma_1\tau\eta_1\tau)\tau\eta_2\tau.
    \end{equation}
    We proceed as in step (2) and create the corresponding vertices over $p_2$. Moreover, we proceed inductively for $\eta_i\eta_{i-1}\dots\eta_1\sigma_1\tau\eta_1\cdots\eta_i\tau$ until $i=b$. For $i=b$, we obtain ends that are labeled by the cycles of $\sigma_2$ and that project to $[p_b,\infty)$. This gives a map between graphs $\tilde{\pi}\colon\tilde{\Gamma_1}\to\Gamma_2$.
    \item The conjugation by $\tau$ induces a natural involution $
    \tilde{\iota}\colon\tilde{\Gamma}_1\to\tilde{\Gamma}_1$ that respects the map $\tilde{\pi}$.
    \item Next, for each edge $e$ of $\tilde{\Gamma}_1$, we replace the cycle labeling it by the length of this cycle. We consider this cycle length as the weight of $e$. 
\end{enumerate}
\end{construction}

\begin{lemma}\label{lem-outcomeconstruction}
Let $(\sigma_1,
\eta_1,\dots,\eta_b,\sigma_2)\in C_g(\mu,\nu)$. Perform Construction \ref{con-twist} to obtain $\pi\colon\Gamma_1\to\Gamma_2$. Then $\pi\colon\Gamma_1\to\Gamma_2$ is indeed a twisted tropical cover.
    \end{lemma}

\begin{proof}
    This follows from the fact that the lengths of cycles (which become weights of the cover) in the cut-and-join analysis add up as expected in the harmonicity condition. The involution is obtained from the action of $\tau$. The connectedness of $\Gamma_1$ follows from the transitiviy condition in Definition \ref{def:twisthurdou} (5).
\end{proof}

\begin{remark}
    
    Given a twisted tropical cover, we can label the edges recursively by cycles of a permutation $\sigma_1$ resp.\ of permutations of the form $\eta_1\sigma_1\tau\eta_1\tau$ and so on. In the following, we study how many ways there are to label the edges of a cover. This builds on \cite[Theorem 2.12]{BF21} which studies the cut-and-join analysis for the action of two partner transpositions. With this, we can determine the number of labels for the edges on the right of a vertex, if the edges on the left are already labeled. We obtain a recursive way to count the number of tuples that yield a given twisted tropical cover. This is detailed in Proposition \ref{thm-covermult}. Before, we count the number of ways to label the left ends of a twisted tropical cover.
   
\end{remark}

The following is a well-known fact (see e.g.\ \cite{goulden1996connection}, Proposition 5.2). For the sake of completeness, we nevertheless include the short proof below.
\begin{lemma}
\label{lem-conjclass}
Let $\mu$ be a partition of $n$.
The cardinality of $B^\sim_\mu$ equals
$$ \frac{1}{2^{\ell(\mu)}} \frac{1}{\prod \mu_i} \frac{1}{|\Aut(\mu)|}(2n)!! .$$
\end{lemma}

\begin{proof}
We consider an ''empty'' permutation consisting of $2$ cycles of length $\mu_i$ for every $i$ such that each cycle contains $\mu_i$ placeholders. Then we count the possibilities to fill up the placeholders one by one with numbers from $1,\ldots,2n$.

For the first place, we have $2n$ options. 
Having chosen one placeholder to be $j$, the element $\tau(j)$ is not allowed to be in the same cycle. For that reason, if our cycle is not yet completed, we have $2n-2$ options to choose for the next place. If the cycle is completed, the involution $\tau$ requires us to fill up the partner cycle of length $1$ as $(\tau(j))$. Thus, we cannot choose $\tau(j)$ either as we fill up the first placeholder of the next cycle. We thus have $2n-2$ options again.

Recursively, we can see that we have $2n!!$ ways to fill our placeholders with numbers. This is not the right count yet, we have overcounted in multiple ways:
\begin{itemize}
    \item For every cycle and its partner cycle, it does not matter in which way we arranged the cycles, so we have to divide by a factor of $2$ of each entry of $\mu$, i.e.\ by $\frac{1}{2^{\ell(\mu)}}$.
    \item If $\mu$ has automorphisms, using the same argument, it does not matter how we arranged, so we have to divide by $\Aut(\mu)$.
    \item Finally, we overcounted for each cycle a factor of $\mu_i$, as we can cyclically exchange the entries in a cycle without changing the cycle.
    
\end{itemize}
Altogether, we obtain the claimed expression.

\end{proof}

\begin{example}
Let $n=3$ and $\mu=(2,1)$. Then $\tau=(14)(25)(36)$.
The expression above yields $\frac{1}{2^2}\frac{1}{2\cdot 1}6!! = 6$.
The six elements in $B^{\sim}_\mu$ are

$$
    (12)(3)(45)(6), (13)(2)(46)(5),
    (15)(3)(24)(6),
    (16)(2)(34)(5),
    (23)(1)(56)(4),
    (26)(1)(53)(4).
$$

\noindent
Let $n=3$ and $\mu=(3)$. The expression yields $\frac{1}{2}\frac{1}{3}6!!= 8$.
The eight elements in $B^{\sim}_\mu$ are

$$
    (123)(654), (132)(564), (126)(354), (162)(534), (135)(264), (153)(624), (156)(324), (165)(234).
$$

\noindent
Let $n=2$ and $\mu=2$. The $\tau=(13)(24)$. The $\frac{1}{2}\frac{1}{2}4!!=2$ elements in $B^{\sim}_\mu$ are

$$ (12)(34), (14)(23).$$

\end{example}

We are now ready to enumerate twisted factorisations that give rise to the same twisted tropical cover.

\begin{proposition}\label{thm-covermult}
Given a twisted tropical cover $\pi\colon\Gamma\to\mathbb{R}$ with involution $\iota$ of type $(g,\mu,\nu)$. Then, the numbers of twisted factorisations in $C_g(\mu,\nu)$ giving rise to $(\pi,\iota)$ via \cref{con-twist} is
\begin{equation}
(2n)!!\cdot 2^{b}\cdot \prod_V (\omega_V-1)\prod_e\omega(e)\cdot \frac{1}{|\Aut(\pi)|},
\end{equation}
where $b$ is the number of branch points. Here, the first product goes over all $4$-valent vertices of $\Gamma$ and $\omega_V$ denotes the weight of the adjacent edges, while the second product is taken over all internal edges of $\Gamma/\iota$ and $\omega(e)$ denotes their weight.
\end{proposition}

\begin{proof}
We count the numbers of twisted factorisations $(\sigma_1,\eta_1,\dots,\eta_b,\sigma_2)\in C_g(\mu,\nu)$ that give rise to $(\pi,\iota)$ via \cref{con-twist}. To begin with, we count the number of first permutations $\sigma$ that may be assigned to the incoming ends. By \cref{lem-conjclass}, this number is equal to
\begin{equation}
    \frac{1}{2^{\ell(\mu)}} \frac{1}{\prod \mu_i} \frac{1}{|\Aut(\mu)|}(2n)!! .
\end{equation}
We fix such a $\sigma$  and follow \cref{con-twist}. First, we label the edges over $[-\infty,p_1]$ by the cycles of $\sigma$ that is in accordance with $\iota$. There are $2^{\ell(\mu)}\cdot |\mathrm{Aut}(\mu)|$ ways to do this: 
First, each edge can be labeled with a cycle or its conjugation under $\tau$, leading to $2^{\ell(\mu)}$ possibilities for labeling.
Furthermore, if we have incoming edges of the same weight, i.e.\ automorphisms of $\mu$, there are a priori $|\Aut(mu)|$ possibilities to permute the cycles for these edges; except if two incoming edges end at the same vertex. Then these two edges are not distinguishable. We nevertheless temporarily record a factor of $2^{\ell(\mu)}\cdot |\mathrm{Aut}(\mu)|$ and keep in mind that we overcounted with a factor of $2$ for each pair  of incoming ends adjacent to the same vertex. The overcounting factors will be cancelled later, when we divide by $|\Aut(\pi)|$, as described below.

We now follow the cut-and-join analysis of \cite[Theorem 2.12]{BF21} to count the number of transpositions that give rise to $(\Gamma,\iota)$ by starting with $\sigma$. We consider the first branch point $p_1$. There are three cases by \cref{con-twist}.\\
In the first case, we have two $3$-valent vertices mapping to $p_1$ that perform a simultaneous join of two pairs of edges $e_1,e_2$ adjacent to the first vertex and $e_3,e_4$ adjacent to the second vertex. Given such a pair of vertices and fixing cycles $\sigma_k^l$ for the edges left of the vertices, there are $2\ell(\sigma_1^i)\ell(\sigma_1^j)=2\omega(e_1)\omega(e_2)= 2\omega(e_3)\omega(e_4)$ transposition $\eta_1$ joining $\sigma_1^i$ with $\sigma_1^j$ and the other two cycles with each other, where we use the notation of Construction \ref{con-twist}.

In the second case, we again have two $3$-valent vetices mapping to $p_1$ that perform a simultaneous cut of two edges $e_1$ and $e_2$ into $e_3,e_4$ and $e_5,e_6$ respectively.  Using again the notation of Construction \ref{con-twist} and \cite[Theorem 2.12]{BF21}, we obtain the following count of transpositions for fixed cycles corresponding to the edges left of the vertices: If $\ell(\sigma_1^{1,a})\neq\ell(\sigma_1^{1,b})$, there are $2\ell(\sigma_1^1)$ many transpositions $\eta_1$ giving rise to this cut. If $\ell(\sigma_1^{1,a})=\ell(\sigma_1^{1,b})$, this number is $\ell(\sigma_1^1)$.
    Since the lengths of the cycles are turned into the weights of the edges of the twisted tropical cover, this can be rephrased as follows:    If $\omega(e_3)\neq\omega(e_4)$, there are $2\omega(e_1)=2\omega(e_2)$ corresponding transpositions $\eta_1$. If $\omega(e_3)=\omega(e_4)$, this number is $\omega(e_1)=\omega(e_2)$. If $e_3$ and $e_4$ have distinguishable evolution in the tropical cover, it matters which cycle takes which path. Therefore, in this case, we have a further contribution of a factor $2$. If the evolution of $e_3$ and $e_4$ is identical, it does not matter which cycle takes which path. Nevertheless, we momentarily insist on overcounting with a factor of $2$ as usual, but remark that an identical evolution     corresponds to a factor of $\mathbb{Z}_2$ in the automorphism group of the tropical cover. Hence, dividing by $\frac{1}{|\Aut(\pi)|}$ later will cancel the overcounting with a factor of $2$ also in the case of identical evolution.\\
In the third case, we have a $4$-valent vertex that performs a join and cut of two incoming edges to two outgoing edges all of the same weight. Using the notation of Construction \ref{con-twist} and \cite[Theorem 2.12]{BF21}, we obtain the following count of transpositions for fixed cycles corresponding to the edges left of the vertices:
There are $\ell(\sigma_1^1)(\ell(\sigma_1^1)-1)$ transpositions $\eta_1$ giving rise to this picture. Since the lengths of the cycles are turned into the weights of the edges of the twisted tropical cover, this can again be rephrased as follows:
There are $\omega(e)(\omega(e)-1)$ many transpositions $\eta_1$ to account for, where $e$ is one of the edges adjacent to the vertex. As in the second case, it matters which cycle corresponding to an outgoing edge takes which path. Thus, we obtain a factor $2$. Again, if the evolution of the two outgoing edges is identical, this corresponds to a non-trivial automorphism of the tropical cover, so the overcounting by $2$ will be cancelled later.\\
Next, we consider the automorphisms of the tropical cover. We have already discussed that two edges obtained from a cut resp.\ from a $4$-valent vertex with indistinguishable evolution correspond to factors of $\mathbb{Z}_2$ in the  automorphism  group.

Since we have only $3$ and $4$-valent vertices (with two incoming and two outgoing edges), any non-trivial automorphism  must exchange pairs of edges with the same image. Accordingly, the automorphism group is a product of factors of $\mathbb{Z}_2$. If we follow the evolution of our cover from the left to the right, then any pair of edges which can lead to a factor of $\mathbb{Z}_2$ comes from a cut or a $4$-valent vertex whose outgoing edges have indistinguishable evolution, except of pairs of incoming ends. The  factors coming from a cut or a $4$-valent vertex are already taken into account, since they cancel with our overcounting for the vertex contributions from above. We only need to discuss pairs of incoming edges which lead to a factor of $\mathbb{Z}_2$ in the automorphism group, because they have the same end vertex.
But these factors are canceled because of the overcounting we performed when we counted the possibilities to label the incoming edges, as mentioned before.\\
To summarise, we obtain the the desired number of twisted factorisations giving rise to $(\pi,\iota)$.
\end{proof}

As the number of tuples in the symmetric group yielding a fixed twisted tropical cover which we determined in Proposition \ref{thm-covermult} exactly equals the multiplicity with which we count a twisted tropical cover in the definition of tropical twisted double Hurwitz number (see Definition \ref{def:twisthurdou}) (up to the global factor $\frac{1}{(2n)!!}$), and since all twisted tropical covers of type $(g,\mu,\nu)$ arise in this manner, we have now proved the following theorem, which we consider the main result of this section.

\begin{theorem}[Correspondence Theorem]
\label{thm-corres}
For $g$ a non-negative integer and $\mu,\nu$ partitions of the same positive integer, the twisted double Hurwitz number of Definition \ref{def:twisthurdou} coincides with the tropical twisted double Hurwitz number of Definition \ref{def:twisttrophn}:
$$\tilde{h}_g(\mu,\nu)=\tilde{h}^{\trop}_g(\mu,\nu).$$
\end{theorem}

As shown in the next chapters, the Correspondence Theorem \ref{thm-corres} allows to use a graphical way of determining twisted Hurwitz numbers, and it can be applied to deduce structural properties such as piecewise polynomiality.

\section{A graphical count of twisted Hurwitz numbers}
\label{sec:monodr}
We distill the combinatorial essence of the count of twisted tropical covers: the metric of the tropical source curve is imposed by the distances of the branch points in the image. Thus, length data does not have to be specified and we can merely count the combinatorial types of tropical twisted covers. This is specified in terms of monodromy graphs.

We can thus use the tropical approach to provide an easy graphical way to determine twisted double Hurwitz numbers.

\begin{definition}[Twisted monodromy graphs]
\label{def-mongraph}
For fixed $g$ and two partitions $\mu$ and $\nu$ of $n$, a graph $\Gamma$ is a \emph{twisted monodromy graph of type $(g, \mu,\nu)$}  if:
\begin{enumerate}
\item $\Gamma$ is a connected, genus $g$, directed graph. 
\item $\Gamma$ has $3$-valent and $4$-valent vertices.
\item $\Gamma$ has an involution whose fixed points are the $4$-valent vertices.
\item $\Gamma$ has $2\ell(\mu)$ inward ends of weights $\mu_i$. Ends which are mapped to each other under the involution have the same weight.
\item $\Gamma$ has $2\ell(\nu)$ outward ends of weights $\nu_i$. Ends mapped to each other via the involution have the same weight.
\item $\Gamma$ does not have sinks or sources.
\item The inner vertices are ordered compatibly with the partial ordering induced by the directions of the edges.
\item Every bounded edge $e$ of the graph is equipped with a weight $w(e)\in \mathbb{N}$. These satisfy the \emph{balancing condition} at each inner vertex: the sum of all weights of  incoming  edges equals the sum of the weights of all outgoing edges.
\item The four edges adjacent to a $4$-valent vertex all have the same weight.
\item The involution is compatible with the weights.
\end{enumerate}
\end{definition}

\begin{corollary}\label{thm-monodromygraphs}
For genus $g$ and two partitions $\mu$ and $\nu$ of $n$, the twisted Hurwitz number $\tilde{h}_g(\mu,\nu)$ equals the weighted sum of all twisted monodromy graphs of type $(g,\mu,\nu)$, where each graph is weighted as follows. Let $b$ denote the number of branch points. 
For a $4$-valent vertex $V$, let $\omega_V$ denote the weight of its $4$ adjacent edges. For an edge $e$, let $\omega(e)$ denote its weight. 

Then we have

$$\tilde{h}_g(\mu,\nu)=\sum_{\Gamma} 2^{b}\cdot \prod_V (\omega_V-1)\prod_e\omega(e)\cdot \frac{1}{|\Aut(\pi)|},$$
where the sum goes over all twisted monodromy graphs $\Gamma$ of type $(g,\mu,\nu)$, the first product goes over all $4$-valent vertices of $\Gamma$ and the second over all pairs of internal edges of $\Gamma$, paired up by the involution.

\end{corollary}

\begin{proof}
By mapping the vertices of a twisted monodromy graph to a line according to their order, each monodromy graph gives rise to a twisted tropical cover. The genus $g$ and the labeling $\mu$, $\nu$ of the ends are given by the type of the monodromy graph. Vice versa, each twisted tropical cover of genus $g$ with ends labeled by $\mu$ and $\nu$ gives rise to a monodromy graph of type $(g,\mu\nu)$ by forgetting lengths of edges and the map to the line. Thus the statement follows from Theorem \ref{thm-corres}. 

\end{proof}

\begin{corollary}\label{cor-monodromygraphs}
We can simplify the count in Corollary \ref{thm-monodromygraphs} as
$$\tilde{h}_g(\mu,\nu)=\sum_{\Gamma} o(\Gamma) \cdot 2^{b}\cdot \prod_V (\omega_V-1)\prod_e\omega(e)\cdot \frac{1}{|\Aut(\pi)|},$$
where now the sum goes over all twisted monodromy graphs without a fixed vertex ordering and $o(\Gamma)$ denotes the number of vertex orderings which are compatible with the edge directions.
\end{corollary}

\begin{example}
We compute $\tilde{h}_1((4),(2,2))=160$.

\begin{figure}
    \centering

\tikzset{every picture/.style={line width=0.75pt}} 

\begin{tikzpicture}[scale=0.7,x=0.75pt,y=0.75pt,yscale=-1,xscale=1]

\draw  [dash pattern={on 0.84pt off 2.51pt}]  (14.65,68.89) -- (313.15,66.41) ;
\draw [line width=1.5]    (27.69,40.37) .. controls (68.29,38.83) and (109.79,51.17) .. (124.31,68.58) ;
\draw [line width=1.5]    (30.76,100.18) .. controls (70.67,98.64) and (111.78,87.42) .. (124.31,68.58) ;
\draw [line width=1.5]    (124.31,68.58) .. controls (143.42,38.2) and (173.17,36.14) .. (192,68.02) ;
\draw [line width=1.5]    (124.31,68.58) .. controls (144.44,97.42) and (172.21,97.8) .. (192,68.02) ;
\draw [line width=1.5]    (192,68.02) .. controls (205.22,49.78) and (217.88,37.58) .. (246.3,37.35) ;
\draw [line width=1.5]    (192,68.02) .. controls (205.86,86.64) and (218.92,98.01) .. (247.35,97.78) ;
\draw [line width=1.5]    (246.3,37.35) .. controls (251.51,25.22) and (265.62,19.06) .. (291.33,18.24) ;
\draw [line width=1.5]    (246.3,37.35) .. controls (251.93,49.39) and (264.89,55.33) .. (291.97,55.1) ;
\draw [line width=1.5]    (247.35,97.78) .. controls (252.56,85.65) and (266.67,79.49) .. (292.38,78.67) ;
\draw [line width=1.5]    (247.35,97.78) .. controls (252.97,109.82) and (265.94,115.76) .. (293.01,115.53) ;
\draw  [dash pattern={on 0.84pt off 2.51pt}]  (351.21,225.06) -- (643.18,225.06) ;
\draw [line width=1.5]    (364.45,198.93) .. controls (404.18,197.82) and (444.56,209.5) .. (458.47,225.62) ;
\draw [line width=1.5]    (366.44,253.98) .. controls (405.5,252.86) and (445.89,242.86) .. (458.47,225.62) ;
\draw [line width=1.5]    (458.47,225.62) .. controls (471.05,210.05) and (484.29,203.38) .. (519.38,203.94) ;
\draw [line width=1.5]    (458.47,225.62) .. controls (469.72,242.86) and (486.27,248.97) .. (519.38,247.86) ;
\draw [line width=1.5]    (519.38,203.94) .. controls (531.29,182.25) and (584.92,170.02) .. (628.62,170.02) ;
\draw [line width=1.5]    (519.38,203.94) -- (576.31,225.06) ;
\draw [line width=1.5]    (519.38,247.86) -- (576.31,225.06) ;
\draw [line width=1.5]    (519.38,247.86) .. controls (530.63,268.99) and (586.24,275.1) .. (629.94,275.1) ;
\draw [line width=1.5]    (576.31,225.06) .. controls (590.22,209.5) and (598.82,203.38) .. (629.94,203.94) ;
\draw [line width=1.5]    (576.31,225.06) .. controls (589.55,242.3) and (598.82,246.75) .. (629.94,247.3) ;
\draw  [dash pattern={on 0.84pt off 2.51pt}]  (340.64,66.51) -- (643.75,66.51) ;
\draw [line width=1.5]    (354.38,40.39) -- (409.37,40.39) ;
\draw [line width=1.5]    (354.38,95.41) -- (409.37,95.41) ;
\draw [line width=1.5]    (409.37,40.39) -- (471.91,67.07) ;
\draw [line width=1.5]    (409.37,95.41) -- (471.91,67.07) ;
\draw [line width=1.5]    (471.91,67.07) -- (517.62,52.62) ;
\draw [line width=1.5]    (409.37,40.39) .. controls (431.36,29.83) and (554.39,50.95) .. (555.08,67.07) ;
\draw [line width=1.5]    (409.37,95.41) .. controls (436.86,103.75) and (555.08,84.3) .. (555.08,67.07) ;
\draw [line width=1.5]    (471.91,67.07) -- (517.62,82.07) ;
\draw [line width=1.5]    (527.93,49.28) -- (629.66,17.05) ;
\draw [line width=1.5]    (526.9,85.96) -- (628.62,118.19) ;
\draw [line width=1.5]    (555.08,67.07) -- (630,90.41) ;
\draw [line width=1.5]    (555.08,67.07) -- (630,44.84) ;
\draw  [dash pattern={on 0.84pt off 2.51pt}]  (8.97,379.02) -- (321,379.02) ;
\draw [line width=1.5]    (23.12,341.8) -- (79.72,341.8) ;
\draw [line width=1.5]    (23.12,420.2) -- (79.72,420.2) ;
\draw [line width=1.5]    (79.72,341.8) .. controls (108.03,318.04) and (151.89,333.09) .. (157.56,349.72) ;
\draw [line width=1.5]    (79.72,420.2) .. controls (108.73,444.76) and (151.89,427.33) .. (157.56,413.08) ;
\draw [line width=1.5]    (118.64,384.96) -- (79.72,420.2) ;
\draw [line width=1.5]    (157.56,349.72) -- (129.96,375.06) ;
\draw [line width=1.5]    (79.72,341.8) -- (157.56,413.08) ;
\draw [line width=1.5]    (157.56,349.72) -- (221.94,349.72) ;
\draw [line width=1.5]    (157.56,413.08) -- (221.94,413.08) ;
\draw [line width=1.5]    (220.96,349.98) .. controls (239.12,339.16) and (250.37,334.91) .. (291.01,335.29) ;
\draw [line width=1.5]    (220.96,349.98) .. controls (238.26,361.95) and (250.37,365.04) .. (291.01,365.43) ;
\draw [line width=1.5]    (220.96,413.34) .. controls (239.12,402.52) and (250.37,398.27) .. (291.01,398.65) ;
\draw [line width=1.5]    (220.96,413.34) .. controls (238.26,425.31) and (250.37,428.4) .. (291.01,428.79) ;
\draw  [dash pattern={on 0.84pt off 2.51pt}]  (13.65,225.49) -- (319.39,225.49) ;
\draw [line width=1.5]    (27.52,199.92) -- (82.98,199.92) ;
\draw [line width=1.5]    (27.52,253.79) -- (82.98,253.79) ;
\draw [line width=1.5]    (82.98,199.92) -- (144.68,226.04) ;
\draw [line width=1.5]    (82.98,253.79) -- (144.68,226.04) ;
\draw [line width=1.5]    (82.98,199.92) -- (296.52,176.52) ;
\draw [line width=1.5]    (82.98,253.79) -- (295.82,273.92) ;
\draw [line width=1.5]    (144.68,226.04) .. controls (179.35,208.62) and (207.08,209.17) .. (234.81,225.49) ;
\draw [line width=1.5]    (144.68,226.04) .. controls (179.35,241.82) and (208.47,242.36) .. (234.81,225.49) ;
\draw [line width=1.5]    (234.81,225.49) -- (296.52,253.25) ;
\draw [line width=1.5]    (234.81,225.49) -- (297.21,197.74) ;
\draw  [dash pattern={on 0.84pt off 2.51pt}]  (351.47,378.98) -- (659,378.98) ;
\draw [line width=1.5]    (365.42,344.49) -- (421.2,344.49) ;
\draw [line width=1.5]    (365.42,417.14) -- (421.2,417.14) ;
\draw [line width=1.5]    (421.2,344.49) -- (635.99,312.94) ;
\draw [line width=1.5]    (421.2,417.14) -- (635.29,444.29) ;
\draw [line width=1.5]    (421.2,344.49) -- (489.54,362.84) ;
\draw [line width=1.5]    (421.2,417.14) -- (490.24,395.13) ;
\draw [line width=1.5]    (489.54,362.84) .. controls (507.67,348.16) and (534.87,345.23) .. (557.88,346.7) ;
\draw [line width=1.5]    (490.24,395.13) .. controls (510.46,407.6) and (534.87,412.74) .. (559.28,410.54) ;
\draw [line width=1.5]    (490.24,395.13) -- (557.88,346.7) ;
\draw [line width=1.5]    (489.54,362.84) -- (507.67,375.31) ;
\draw [line width=1.5]    (518.83,383.39) -- (559.28,410.54) ;
\draw [line width=1.5]    (557.88,346.7) -- (634.59,346.7) ;
\draw [line width=1.5]    (559.28,410.54) -- (635.99,410.54) ;
\draw  [dash pattern={on 0.84pt off 2.51pt}]  (18.58,545.3) -- (305.98,545.3) ;
\draw [line width=1.5]    (31.61,515.26) -- (83.75,515.26) ;
\draw [line width=1.5]    (31.61,578.53) -- (83.75,578.53) ;
\draw [line width=1.5]    (83.75,515.26) .. controls (109.82,496.09) and (114.38,496.09) .. (146.97,496.09) ;
\draw [line width=1.5]    (146.97,496.09) -- (292.29,494.17) ;
\draw [line width=1.5]    (146.97,496.09) -- (220.61,520.37) ;
\draw [line width=1.5]    (220.61,520.37) -- (290.34,520.37) ;
\draw [line width=1.5]    (83.75,578.53) .. controls (108.52,597.7) and (115.68,597.06) .. (148.27,597.06) ;
\draw [line width=1.5]    (148.27,597.06) -- (291.64,596.42) ;
\draw [line width=1.5]    (148.27,597.06) -- (219.3,572.14) ;
\draw [line width=1.5]    (219.3,572.14) -- (290.99,572.14) ;
\draw [line width=1.5]    (83.75,515.26) -- (219.3,572.14) ;
\draw [line width=1.5]    (83.75,578.53) -- (152.18,549.45) ;
\draw [line width=1.5]    (170.43,541.46) -- (220.61,520.37) ;
\draw  [dash pattern={on 0.84pt off 2.51pt}]  (352.84,543.8) -- (648.11,543.8) ;
\draw [line width=1.5]    (366.23,512.43) -- (419.79,512.43) ;
\draw [line width=1.5]    (366.23,578.51) -- (419.79,578.51) ;
\draw [line width=1.5]    (419.79,512.43) .. controls (446.57,492.4) and (485.41,497.74) .. (491.43,517.1) ;
\draw [line width=1.5]    (419.79,578.51) .. controls (445.24,598.54) and (484.74,591.19) .. (492.1,572.5) ;
\draw [line width=1.5]    (419.79,512.43) -- (492.1,572.5) ;
\draw [line width=1.5]    (419.79,578.51) -- (455.61,547.81) ;
\draw [line width=1.5]    (465.32,539.8) -- (491.43,517.1) ;
\draw [line width=1.5]    (491.43,517.1) -- (573.79,517.1) ;
\draw [line width=1.5]    (492.1,572.5) -- (574.46,572.5) ;
\draw [line width=1.5]    (573.79,517.1) .. controls (590.53,497.74) and (607.93,498.41) .. (639.4,497.08) ;
\draw [line width=1.5]    (574.46,572.5) .. controls (587.18,593.2) and (611.28,589.19) .. (640.07,591.86) ;
\draw [line width=1.5]    (574.46,572.5) .. controls (591.2,553.15) and (608.6,552.48) .. (640.07,551.14) ;
\draw [line width=1.5]    (573.79,517.1) .. controls (586.51,537.79) and (610.61,533.79) .. (639.4,536.46) ;

\draw (10.41,29.93) node [anchor=north west][inner sep=0.75pt]  [rotate=-359.3]  {$4$};
\draw (16.41,92.32) node [anchor=north west][inner sep=0.75pt]  [rotate=-359.3]  {$4$};
\draw (149.2,97.39) node [anchor=north west][inner sep=0.75pt]  [rotate=-359.3]  {$4$};
\draw (151.04,25) node [anchor=north west][inner sep=0.75pt]  [rotate=-359.3]  {$4$};
\draw (206.26,24.02) node [anchor=north west][inner sep=0.75pt]  [rotate=-359.3]  {$4$};
\draw (212.24,95.33) node [anchor=north west][inner sep=0.75pt]  [rotate=-359.3]  {$4$};
\draw (298.27,10.04) node [anchor=north west][inner sep=0.75pt]  [rotate=-359.3]  {$2$};
\draw (301.68,46.82) node [anchor=north west][inner sep=0.75pt]  [rotate=-359.3]  {$2$};
\draw (300.5,71.06) node [anchor=north west][inner sep=0.75pt]  [rotate=-359.3]  {$2$};
\draw (299.9,108.53) node [anchor=north west][inner sep=0.75pt]  [rotate=-359.3]  {$2$};
\draw (345.04,191.59) node [anchor=north west][inner sep=0.75pt]    {$4$};
\draw (345.31,248.5) node [anchor=north west][inner sep=0.75pt]    {$4$};
\draw (482.26,249.8) node [anchor=north west][inner sep=0.75pt]    {$4$};
\draw (483.58,184.78) node [anchor=north west][inner sep=0.75pt]    {$4$};
\draw (633.87,162.98) node [anchor=north west][inner sep=0.75pt]    {$2$};
\draw (639.6,275.34) node [anchor=north west][inner sep=0.75pt]    {$2$};
\draw (635.2,239.7) node [anchor=north west][inner sep=0.75pt]    {$2$};
\draw (634.53,197.45) node [anchor=north west][inner sep=0.75pt]    {$2$};
\draw (340.82,32.24) node [anchor=north west][inner sep=0.75pt]    {$4$};
\draw (341.51,87.81) node [anchor=north west][inner sep=0.75pt]    {$4$};
\draw (476.94,101.77) node [anchor=north west][inner sep=0.75pt]    {$2$};
\draw (478.29,20.93) node [anchor=north west][inner sep=0.75pt]    {$2$};
\draw (634.31,4.45) node [anchor=north west][inner sep=0.75pt]    {$2$};
\draw (633.62,111.15) node [anchor=north west][inner sep=0.75pt]    {$2$};
\draw (635.68,81.14) node [anchor=north west][inner sep=0.75pt]    {$2$};
\draw (635,38.91) node [anchor=north west][inner sep=0.75pt]    {$2$};
\draw (414.49,47.4) node [anchor=north west][inner sep=0.75pt]    {$2$};
\draw (413.98,70.27) node [anchor=north west][inner sep=0.75pt]    {$2$};
\draw (9.34,333.41) node [anchor=north west][inner sep=0.75pt]    {$4$};
\draw (10.04,412.6) node [anchor=north west][inner sep=0.75pt]    {$4$};
\draw (299.01,357.04) node [anchor=north west][inner sep=0.75pt]    {$2$};
\draw (298.14,327.68) node [anchor=north west][inner sep=0.75pt]    {$2$};
\draw (301.13,421.98) node [anchor=north west][inner sep=0.75pt]    {$2$};
\draw (300.27,392.62) node [anchor=north west][inner sep=0.75pt]    {$2$};
\draw (102.82,309.18) node [anchor=north west][inner sep=0.75pt]    {$2$};
\draw (108.43,344.82) node [anchor=north west][inner sep=0.75pt]    {$2$};
\draw (105.69,400.82) node [anchor=north west][inner sep=0.75pt]    {$2$};
\draw (98.61,438.58) node [anchor=north west][inner sep=0.75pt]    {$2$};
\draw (188.47,327.68) node [anchor=north west][inner sep=0.75pt]    {$4$};
\draw (189.41,414.19) node [anchor=north west][inner sep=0.75pt]    {$4$};
\draw (13.89,191.77) node [anchor=north west][inner sep=0.75pt]    {$4$};
\draw (14.59,246.19) node [anchor=north west][inner sep=0.75pt]    {$4$};
\draw (299.53,168.37) node [anchor=north west][inner sep=0.75pt]    {$2$};
\draw (300.93,193.58) node [anchor=north west][inner sep=0.75pt]    {$2$};
\draw (300.22,245.1) node [anchor=north west][inner sep=0.75pt]    {$2$};
\draw (300.91,267.41) node [anchor=north west][inner sep=0.75pt]    {$2$};
\draw (127.55,199.94) node [anchor=north west][inner sep=0.75pt]    {$2$};
\draw (126.15,236.09) node [anchor=north west][inner sep=0.75pt]    {$2$};
\draw (204.33,195.65) node [anchor=north west][inner sep=0.75pt]    {$2$};
\draw (213.95,239.6) node [anchor=north west][inner sep=0.75pt]    {$2$};
\draw (351.74,336.16) node [anchor=north west][inner sep=0.75pt]    {$4$};
\draw (352.44,409.54) node [anchor=north west][inner sep=0.75pt]    {$4$};
\draw (637.94,307.26) node [anchor=north west][inner sep=0.75pt]    {$2$};
\draw (640.72,434.09) node [anchor=north west][inner sep=0.75pt]    {$2$};
\draw (636.26,339.83) node [anchor=north west][inner sep=0.75pt]    {$2$};
\draw (638.75,396.44) node [anchor=north west][inner sep=0.75pt]    {$2$};
\draw (441.7,354.51) node [anchor=north west][inner sep=0.75pt]    {$2$};
\draw (440.31,387.24) node [anchor=north west][inner sep=0.75pt]    {$2$};
\draw (491.21,337.77) node [anchor=north west][inner sep=0.75pt]    {$1$};
\draw (503.07,407.34) node [anchor=north west][inner sep=0.75pt]    {$1$};
\draw (542.82,383.12) node [anchor=north west][inner sep=0.75pt]    {$1$};
\draw (546.59,355.97) node [anchor=north west][inner sep=0.75pt]    {$1$};
\draw (18.45,507.02) node [anchor=north west][inner sep=0.75pt]    {$4$};
\draw (19.1,570.93) node [anchor=north west][inner sep=0.75pt]    {$4$};
\draw (293.46,485.93) node [anchor=north west][inner sep=0.75pt]    {$2$};
\draw (292.81,512.77) node [anchor=north west][inner sep=0.75pt]    {$2$};
\draw (292.16,565.18) node [anchor=north west][inner sep=0.75pt]    {$2$};
\draw (290.86,591.38) node [anchor=north west][inner sep=0.75pt]    {$2$};
\draw (99.26,479.29) node [anchor=north west][inner sep=0.75pt]    {$3$};
\draw (92.99,592.12) node [anchor=north west][inner sep=0.75pt]    {$3$};
\draw (120.11,511.93) node [anchor=north west][inner sep=0.75pt]    {$1$};
\draw (119.46,565.43) node [anchor=north west][inner sep=0.75pt]    {$1$};
\draw (169.64,509.68) node [anchor=north west][inner sep=0.75pt]    {$1$};
\draw (168.04,566.89) node [anchor=north west][inner sep=0.75pt]    {$1$};
\draw (352.86,504.16) node [anchor=north west][inner sep=0.75pt]    {$4$};
\draw (353.53,570.91) node [anchor=north west][inner sep=0.75pt]    {$4$};
\draw (642.78,542.88) node [anchor=north west][inner sep=0.75pt]    {$2$};
\draw (640.77,590.94) node [anchor=north west][inner sep=0.75pt]    {$2$};
\draw (435.89,482.13) node [anchor=north west][inner sep=0.75pt]    {$3$};
\draw (427.18,590.27) node [anchor=north west][inner sep=0.75pt]    {$3$};
\draw (640.77,528.86) node [anchor=north west][inner sep=0.75pt]    {$2$};
\draw (639.43,490.14) node [anchor=north west][inner sep=0.75pt]    {$2$};
\draw (526.95,499.49) node [anchor=north west][inner sep=0.75pt]    {$4$};
\draw (526.95,574.25) node [anchor=north west][inner sep=0.75pt]    {$4$};
\draw (441.91,514.84) node [anchor=north west][inner sep=0.75pt]    {$1$};
\draw (442.58,560.23) node [anchor=north west][inner sep=0.75pt]    {$1$};
\draw (551,197.98) node [anchor=north west][inner sep=0.75pt]    {$2$};
\draw (556,237.98) node [anchor=north west][inner sep=0.75pt]    {$2$};

\end{tikzpicture}

    \caption{The tropical count of $\tilde{h}_1((4),(2,2))$.}
    \label{fig:ExH^1((4),(2,2))1}
\end{figure}
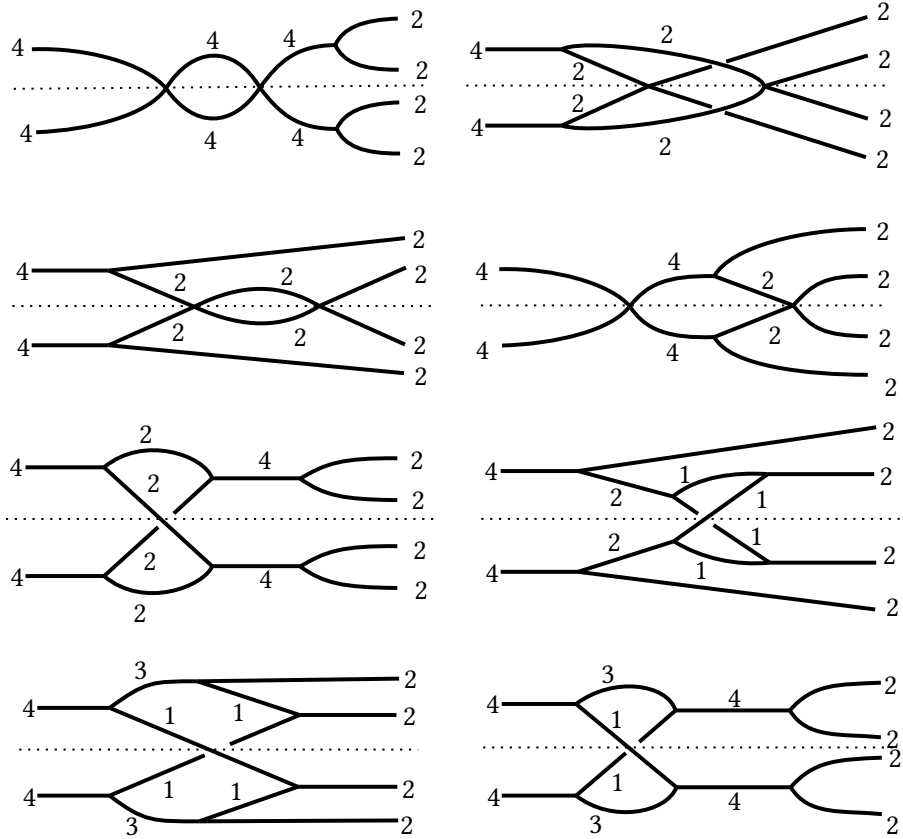

Figure \ref{fig:ExH^1((4),(2,2))1} shows all monodromy graphs of type $(1,(4),(2,2))$. We did not draw directions for the edges, they are all directed to the right. The order of the vertices is always fixed by the direction of the edges, except for the top right graph for which the second and third vertex could be exchanged. As these orderings yield isomorphic covers, we only need to consider one.

There are $3$ branch points. In the top left graph in Figure \ref{fig:ExH^1((4),(2,2))1}, we have two $4$-valent vertices, each adjacent to edges of weight $4$, four (independent) automorphisms and two pairs of inner edges of weight $4$ each. Thus the top graph contributes $2^3 \cdot 3\cdot 3\cdot 4\cdot 4\cdot (\frac{1}{2})^4=72$.

For the top right graph, there are two $4$-valent vertices, each adjacent to edges of weight $2$, three (independent) automorphism and two pairs of inner edges of weight $2$ each. Thus the top right graph contributes $2^3\cdot 1\cdot 1\cdot 2\cdot 2\cdot (\frac{1}{2})^3=4$.

For the left graph in the second row, there are two $4$-valent vertices, each adjacent to edges of weight $2$, three (independent) automorphism and two pairs of inner edges of weight $2$ each. Thus the third graph contributes $2^3\cdot 1\cdot 1\cdot 2\cdot 2\cdot (\frac{1}{2})^3=4$.

The right graph in the second row has two $4$-valent vertices, one adjacent to edges of weight $4$ the other to weight $2$, it has three (independent) automorphisms and two pairs of inner edges of weight $4$ resp.\ $2$. Thus it contributes $2^3\cdot 3\cdot 1\cdot 4\cdot 2\cdot (\frac{1}{2})^3=24$.

The left graph in the third row has no $4$-valent vertices, three (independent) automorphism and three pairs of inner edges of weights $2$, $2$ resp.\ $4$. Thus it contributes $2^3\cdot 2\cdot 2\cdot 4\cdot (\frac{1}{2})^3=16$.

The right graph in the third row has no $4$-valent vertices, two (independent) automorphism, one pairs of inner edges of weight $2$ and two pairs of inner edges of weight $1$. Thus, it contributes $2^3\cdot 1\cdot 1\cdot 2\cdot (\frac{1}{2})^2=4$.

The left graph in the last row has no $4$-valent vertices, one automorphism and three pairs of inner edges of weights $1$, $1$ resp.\ $3$. Thus it contributes $2^3\cdot 1\cdot 1\cdot 3\cdot \frac{1}{2}=12$.

The right graph in the last row has no $4$-valent vertices, two automorphisms and three pairs of inner edges of weights $1$, $3$ resp.\ $4$. Thus it contributes $2^3\cdot 1\cdot 3\cdot 4\cdot (\frac{1}{2})^2=24$.

\end{example}

\section{Piecewise polynomial structure of twisted double Hurwitz numbers}
\label{sec:poly}

As for usual double Hurwitz numbers, the tropical approach can be used to deduce the piecewise polynomial structure.

\begin{theorem}\label{thm-PP}
Fix $g$, $\ell(\mu)$ and $\ell(\nu)$. Consider the twisted double Hurwitz numbers as a function 
$$\tilde{h}_g(\mu,\nu):\Big\{(\mu,\nu)\;|\;\sum \mu_i=\sum\nu_j\Big\} \longrightarrow\mathbb{Q}.$$ 

Then the twisted double Hurwitz numbers $\tilde{h}_g(\mu,\nu)$ are piecewise polynomial in the entries $\mu_i$, $\nu_i$. 

The chambers of polynomiality are given by the hyperplane arrangement $\{\sum_{i\in I}\mu_i=\sum_{j\in J}\nu_j\}$, where $I$ and $J$ are subsets of $[\ell(\mu)]$ resp.\ $[\ell(\nu)]$ of size at least $1$ and at most $\ell(\mu)-1$ resp.\ $\ell(\nu)-1$.

In a chamber, the polynomial $\tilde{h}_g(\mu,\nu)$ is of degree $\ell(\mu)+\ell(\nu)-1+2g$. 

\end{theorem}
Here, since we treat $\mu$ and $\nu$ as variables, we assume implicitely that $\Aut(\mu)$ and $\Aut(\nu)$ are trivial. If one inserts special values $\mu,\nu$ for which this is not the case, one has to divide the piecewise polynomial expression by $\Aut(\mu)\cdot \Aut(\nu)$ to pass to the twisted Hurwitz number.


\begin{proof}
The proof builds on the proof of the piecewise polynomiality for double Hurwitz numbers obtained by tropical methods in Theorem 3.1 and 3.9 in \cite{CJM11}.

Let $\Gamma$ be a twisted monodromy graph of type $(g,\mu,\nu)$. Here, we treat the entries of $\mu$ and $\nu$ as variables, but $g$, $\ell(\mu)$ and $\ell(\nu)$ are fixed numbers. Accordingly, the weights of the edges of $\Gamma$ are also variable and depend on the $\mu_i$ and $\nu_j$. The number $o(\Gamma)$ of vertex orderings compatible with the edge directions, the factor $\frac{1}{|\Aut(\pi)|}$ and the factor $2^{b}$ (where $b$ is the number of branch points) showing up in the multiplicity with which $\Gamma$ contributes to $\tilde{h}_g(\mu,\nu)$ in Corollary \ref{cor-monodromygraphs} are all numbers. Thus it remains to show the piecewise polynomiality of the factors $\prod_V (\omega_V-1)$ and $\prod_e\omega(e)$. Recall that the first factor is the product over all $4$-valent vertices $V$ and $\omega_V$ denotes the weight of its adjacent edges, and the second factor the product over all inner edges $e$ and $\omega(e)$ denotes their weight.

Assume $\Gamma$ has $c$ $4$-valent vertices. By an Euler-characteristics computation, $\Gamma$ has $$2(\ell(\mu)+\ell(\nu))+3g-3-c$$ bounded edges and $$ 2(\ell(\mu)+\ell(\nu))+2g-2-2c$$ $3$-valent vertices. The quotient $\Gamma/\iota$ thus has $$\ell(\mu)+\ell(\nu)+\frac{3g-3-c}{2}$$ bounded edges and $$\ell(\mu)+\ell(\nu)+g-c-1$$ $3$-valent vertices. In addition, it has $c$ $2$-valent vertices. We can temporarily remove those and merge the adjacent edges, arriving at a $3$-valent graph with $\ell(\mu)+\ell(\nu)$ ends and $\ell(\mu)+\ell(\nu)+\frac{3g-3-c}{2}-c$ edges. Again by an Euler characteristic computation, for the genus $g'$ of $\Gamma/\iota$ we obtain
$\ell(\mu)+\ell(\nu)+3g'-3$ bounded edges. Equating the two expressions above, we derive at $g'=\frac{1}{2}\cdot (g-c+1)$.

We pick  $g'$ edges in $\Gamma/\iota$ whose removal produces a connected tree. We denote the weights of these edges by variables $i_1,\ldots,i_{g'}$. Then by the balancing condition, the weights of the remaining edges are fixed and a homogeneous linear polynomial in the $\mu_i$, $\nu_j$ and the variables $i_k$. The condition that the edge weights have to be positive provides summation bounds for the variables $i_k$ forming a bounded polyhedron. We sum the product over the edge weights and the factors $(\omega_V-1)$ for the $4$-valent vertices over the lattice points of this polyhedron. We have $\ell(\mu)+\ell(\nu)+\frac{3g-3-c}{2}$ factors for the inner edges, and $c$ for the $4$-valent vertices. Using the Faulhaber formulae for summation of powers of the variables $i_k$, we increase the degree by one with each summation.
Thus, we obtain a polynomial of degree 

\begin{align}&\ell(\mu)+\ell(\nu)+\frac{3g-3-c}{2}+c+g'
\\&= \ell(\mu)+\ell(\nu)+\frac{3g-3-c}{2}+c+\frac{1}{2}(g-c+1)
\\&= \ell(\mu)+\ell(\nu)+2g-1.\end{align}

As we obtain a polynomial contribution of this degree for every twisted monodromy graph, we deduce that $\tilde{h}_g(\mu,\nu)$ is a polynomial of degree $\ell(\mu)+\ell(\nu)+2g-1$. 

The piecewise polynomial structure and the walls for the chambers of polynomiality arise as in Theorem 3.9 of \cite{CJM11}, which we state below in \ref{thm-walls}.

\end{proof}

\begin{theorem}[Theorem 3.9, \cite{CJM11}]\label{thm-walls}
For the piecewise polynomial function which associates to $\mu,\nu$ the double Hurwitz number counting branched covers of $\mathbb{P}^1$ with two special ramification profiles $\mu$, $\nu$ and only simple ramification else as in \cite{GJV05, CJM11}, the walls separating the chambers  of polynomiality are given by equations of the form $$\sum_{i\in I} \mu_i - \sum_{j\in J} \nu_j=0,$$ where $I$ and $J$ are subsets of $[\ell(\mu)]$ resp.\ $[\ell(\nu)]$ of size at least $1$ and at most $\ell(\mu)-1$ resp.\ $\ell(\nu)-1$.
\end{theorem}
These equations state that an edge of a feasible monodromy graph can have weight $0$. On one side of the wall, the graph with this edge directed in one way appears; on the other side, the graph with the direction of this edge reversed. This is in charge of the \emph{piecewise} polynomial behaviour: we do not sum over the same monodromy graphs. The situation is exactly the same in the case of twisted monodromy graphs.

\begin{remark}
If we consider double Hurwitz numbers counting branched covers of $\mathbb{P}^1$ with two special ramification profiles $\mu$, $\nu$ and only simple ramification else as in \cite{GJV05, CJM11}, we obtain a polynomial in the $\mu_i$, $\nu_j$ of degree $4g-3+\ell(\mu)+\ell(\nu)$.
This fits nicely with the genus of the surface $S$ mentioned in Definition \ref{def-nonoriented}, which equals $\frac{g+1}{2}$.
\end{remark}

\begin{remark}
    Note that \cite[Theorem 6.6]{chapuy2020non} proves piecewise polynomiality for the more general $b$-Hurwitz numbers, which for $b=1$ coincide with our twisted Hurwitz numbers up to a factor of $n$. 
    One could say that \cref{thm-PP} improves this result, as it shows that the polynomials in \cite[Theorem 6.6]{chapuy2020non} are divisible by $n$. This can be seen also directly in the context of \cite[Theorem 6.6]{chapuy2020non} however, as for the special case $b=1$ the factor of $n$ corresponds to a combinatorial factor.
\end{remark}

In the following, we study the special situations of twisted double Hurwitz numbers of genus zero and genus one, as in these cases we can be more precise and include statements not only about the top degree but also about the smallest degree appearing in the polynomials $\tilde{h}_0(\mu,\nu)$ and $\tilde{h}_1(\mu,\nu)$. The next lemma serves as a preparation for the two following propositions dealing with the cases genus zero and genus one.
 
\begin{lemma}\label{lem-g+1}
Let $\Gamma$ be a twisted monodromy graph of type $(g,\mu,\nu)$. Then $\Gamma$ can have at most $g+1$ $4$-valent vertices. 
\end{lemma}
\begin{proof}
The tropical twisted cover induced by a twisted monodromy graph of type $(g,\mu,\nu)$ has $\ell(\mu)+\ell(\nu)+g-1$ branch points, by an Euler characteristics computation for $\Gamma$. If the branch point is given by a $4$-valent vertex, the partition of weights left of the branch point equals the partition of weights right of the branch points. In order to change the partition from $2\mu$ to $2\nu$, we need at least as many branch points which are not $4$-valent as two involuted copies of a rational cover with ends $\mu$ and $\nu$ have, i.e.\ $\ell(\mu)+\ell(\nu)-2$. Thus at most $g+1$ branch points can come from $4$-valent vertices.
\end{proof}
 
 \begin{proposition}\label{prop-ppingenus0}
 Consider the genus zero twisted double Hurwitz numbers as a piecewise polynomial function 
$$\tilde{h}_0(\mu,\nu):\Big\{(\mu,\nu)\;|\;\sum \mu_i=\sum\nu_j\Big\} \longrightarrow\mathbb{Q}.$$ 

In each chamber, the polynomial $\tilde{h}_0(\mu,\nu)$ has two homogeneous components, one of degree $\ell(\mu)+\ell(\nu)-1$ and one of degree $\ell(\mu)+\ell(\nu)-2$.

Moreover, each monodromy graph contributing to the count of  $\tilde{h}_0(\mu,\nu)$ has precisely one $4$-valent vertex.
 \end{proposition}
 
 \begin{proof}
 By Lemma \ref{lem-g+1}, a rational twisted monodromy graph has at most one $4$-valent vertex. By an Euler characteristics computation, $\Gamma$ has $2(\ell(\mu)+\ell(\nu))-3-c$ inner edges, where $c$ is the number of $4$-valent vertices. Because of the involution, this number must be even and thus there is precisely one $4$-valent vertex.
 
 The weight with which $\Gamma$ contributes to the count of $\tilde{h}_0(\mu,\nu)$ by Theorem \ref{thm-monodromygraphs} equals the product of the edge weights of each pair of inner edges times a factor of $(\omega_V-1)$ for the $4$-valent vertex times a rational number.
 Each edge weight is a homogeneous linear polynomial in the $\mu_i$ and $\nu_j$, and thus also the factor $(\omega_V-1)$ is linear in $\mu_i$ and $\nu_j$ (although not homogeneous).
 It follows that $\Gamma$ contributes a polynomial of degree $\frac{1}{2}\cdot (2(\ell(\mu)+\ell(\nu))-3-1)+1= \ell(\mu)+\ell(\nu)-1$. The terms of smallest degree have degree $\ell(\mu)+\ell(\nu)-2$.
 \end{proof}

 \begin{example}
 In this example, we compute the polynomial $\tilde{h}_0((\mu),(\nu_1,\nu_2))$. As the partition $(\mu)$ consists of just one entry, there are no walls and there is no piecewise structure, but there is only one polynomial. The three monodromy graphs which contribute to the count of $\tilde{h}_0((\mu),(\nu_1,\nu_2))$ are depicted in Figure \ref{fig:exppingenus0}. The polynomial with which they contribute is written next to the picture.
 Altogether, we obtain
 $$ \tilde{h}_0((\mu),(\nu_1,\nu_2)) = \mu\cdot (\mu-1)+\nu_1\cdot (\nu_1-1)+\nu_2\cdot (\nu_2-1)= (\mu^2+\nu_1^2+\nu_2^2)-(\mu+\nu_1+\nu_2), $$
 a polynomial with two homogeneous parts, one of degree two and one of degree one, as expected by Proposition \ref{prop-ppingenus0}.

 \begin{figure}
     \centering

\tikzset{every picture/.style={line width=0.75pt}} 

\begin{tikzpicture}[x=0.75pt,y=0.75pt,yscale=-1,xscale=1]

\draw  [dash pattern={on 0.84pt off 2.51pt}]  (90,340) -- (280,340) ;
\draw    (90,320) .. controls (115.25,320.22) and (127.75,322.72) .. (150,340) ;
\draw    (90,360) .. controls (115.25,360.22) and (134.75,354.22) .. (150,340) ;
\draw    (150,340) .. controls (168.75,325.72) and (181.25,319.72) .. (210,320) ;
\draw    (150,340) .. controls (168.25,352.72) and (175.25,359.72) .. (210,360) ;
\draw    (210,320) .. controls (221.25,314.22) and (231.75,310.22) .. (250,310) ;
\draw    (210,360) .. controls (221.25,354.22) and (231.75,350.22) .. (250,350) ;
\draw    (210,360) .. controls (220.75,367.72) and (231.75,370.22) .. (250,370) ;
\draw    (210,320) .. controls (220.75,327.72) and (231.75,330.22) .. (250,330) ;
\draw  [dash pattern={on 0.84pt off 2.51pt}]  (90,440) -- (280,440) ;
\draw    (150,420) .. controls (161.25,414.22) and (231.75,410.22) .. (250,410) ;
\draw    (200,440) .. controls (211.25,434.22) and (231.75,430.22) .. (250,430) ;
\draw    (150,460) .. controls (160.75,467.72) and (231.75,470.22) .. (250,470) ;
\draw    (150,420) .. controls (160.75,427.72) and (189.25,430.22) .. (200,440) ;
\draw    (90,420) -- (150,420) ;
\draw    (90,460) -- (150,460) ;
\draw    (150,460) .. controls (162.75,451.22) and (186.25,450.72) .. (200,440) ;
\draw    (200,440) .. controls (210.75,447.72) and (231.75,450.22) .. (250,450) ;
\draw  [dash pattern={on 0.84pt off 2.51pt}]  (90,540) -- (280,540) ;
\draw    (150,520) .. controls (161.25,514.22) and (231.75,510.22) .. (250,510) ;
\draw    (200,540) .. controls (211.25,534.22) and (231.75,530.22) .. (250,530) ;
\draw    (150,560) .. controls (160.75,567.72) and (231.75,570.22) .. (250,570) ;
\draw    (150,520) .. controls (160.75,527.72) and (189.25,530.22) .. (200,540) ;
\draw    (90,520) -- (150,520) ;
\draw    (90,560) -- (150,560) ;
\draw    (150,560) .. controls (162.75,551.22) and (186.25,550.72) .. (200,540) ;
\draw    (200,540) .. controls (210.75,547.72) and (231.75,550.22) .. (250,550) ;

\draw (72,314.29) node [anchor=north west][inner sep=0.75pt]   [align=left] {$\displaystyle \mu $};
\draw (73.14,355.71) node [anchor=north west][inner sep=0.75pt]   [align=left] {$\displaystyle \mu $};
\draw (179.14,306.57) node [anchor=north west][inner sep=0.75pt]   [align=left] {$\displaystyle \mu $};
\draw (175.29,362) node [anchor=north west][inner sep=0.75pt]   [align=left] {$\displaystyle \mu $};
\draw (260.14,302) node [anchor=north west][inner sep=0.75pt]   [align=left] {$\displaystyle \nu _{1}$};
\draw (260.71,365.14) node [anchor=north west][inner sep=0.75pt]   [align=left] {$\displaystyle \nu _{1}$};
\draw (261,320) node [anchor=north west][inner sep=0.75pt]   [align=left] {$\displaystyle \nu _{2}$};
\draw (260.71,344.86) node [anchor=north west][inner sep=0.75pt]   [align=left] {$\displaystyle \nu _{2}$};
\draw (72.29,411.43) node [anchor=north west][inner sep=0.75pt]   [align=left] {$\displaystyle \mu $};
\draw (74.86,454.57) node [anchor=north west][inner sep=0.75pt]   [align=left] {$\displaystyle \mu $};
\draw (259.57,400.57) node [anchor=north west][inner sep=0.75pt]   [align=left] {$\displaystyle \nu _{1}$};
\draw (259.57,463.71) node [anchor=north west][inner sep=0.75pt]   [align=left] {$\displaystyle \nu _{1}$};
\draw (259.29,420.86) node [anchor=north west][inner sep=0.75pt]   [align=left] {$\displaystyle \nu _{2}$};
\draw (259.86,443.71) node [anchor=north west][inner sep=0.75pt]   [align=left] {$\displaystyle \nu _{2}$};
\draw (181,418.57) node [anchor=north west][inner sep=0.75pt]   [align=left] {$\displaystyle \nu _{2}$};
\draw (181.57,451.14) node [anchor=north west][inner sep=0.75pt]   [align=left] {$\displaystyle \nu _{2}$};
\draw (73.71,513.43) node [anchor=north west][inner sep=0.75pt]   [align=left] {$\displaystyle \mu $};
\draw (73.14,553.14) node [anchor=north west][inner sep=0.75pt]   [align=left] {$\displaystyle \mu $};
\draw (260.71,501.43) node [anchor=north west][inner sep=0.75pt]   [align=left] {$\displaystyle \nu _{2}$};
\draw (259.86,563.43) node [anchor=north west][inner sep=0.75pt]   [align=left] {$\displaystyle \nu _{2}$};
\draw (261,520.29) node [anchor=north west][inner sep=0.75pt]   [align=left] {$\displaystyle \nu _{1}$};
\draw (260.71,542.29) node [anchor=north west][inner sep=0.75pt]   [align=left] {$\displaystyle \nu _{1}$};
\draw (180.14,516.86) node [anchor=north west][inner sep=0.75pt]   [align=left] {$\displaystyle \nu _{1}$};
\draw (180.43,554) node [anchor=north west][inner sep=0.75pt]   [align=left] {$\displaystyle \nu _{1}$};
\draw (350,330) node [anchor=north west][inner sep=0.75pt]   [align=left] {$\displaystyle  \mu ( \mu -1)$};
\draw (351,430) node [anchor=north west][inner sep=0.75pt]   [align=left] {$\displaystyle  \nu _{2}( \nu _{2} -1)$};
\draw (351,530) node [anchor=north west][inner sep=0.75pt]   [align=left] {$\displaystyle \nu _{1}( \nu _{1} -1)$};

\end{tikzpicture}

     \caption{The computation of the polynomial $\tilde{h}_0((\mu),(\nu_1,\nu_2))$.}
     \label{fig:exppingenus0}
 \end{figure}
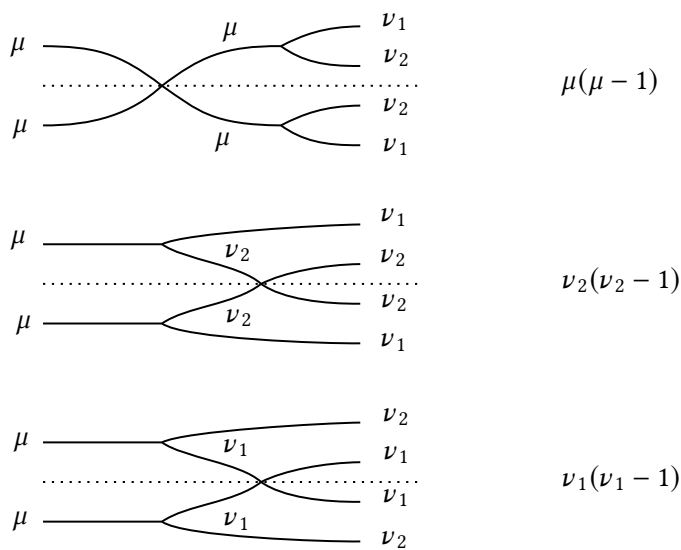

 \end{example}

The following statement will be used in the next Proposition:
\begin{corollary}[Corollary 1.2, \cite{Joh15}]\label{cor-joh}
  The piecewise polynomial function which associates to $\mu,\nu$ the double Hurwitz number counting branched covers of $\mathbb{P}^1$ with two special ramification profiles $\mu$, $\nu$ and only simple ramification else as in \cite{GJV05, CJM11} is, when restricted to a chamber of polynomiality, a sum of homogeneous polynomials of even degrees ranging from $ 4g-3+\ell(\mu)+\ell(\nu)$ to $2g-3+\ell(\mu)+\ell(\nu)$.
\end{corollary}
 
  \begin{proposition}\label{prop-ppingenus1}
 Consider the genus one twisted double Hurwitz numbers as a piecewise polynomial function 
$$\tilde{h}_1(\mu,\nu):\Big\{(\mu,\nu)\;|\;\sum \mu_i=\sum\nu_j\Big\} \longrightarrow\mathbb{Q}.$$ 

In each chamber, the polynomial $\tilde{h}_1(\mu,\nu)$ has three homogeneous components, ranging from degree $\ell(\mu)+\ell(\nu)-1$ to degree $\ell(\mu)+\ell(\nu)+1$.

Moreover, each monodromy graph contributing to the count of  $\tilde{h}_1(\mu,\nu)$ has either none or two $4$-valent vertices.
 \end{proposition}
 
 \begin{proof}
 The monodromy graph $\Gamma$ has $2(\ell(\mu)+\ell(\nu))+3g-3-c=2(\ell(\mu)+\ell(\nu))-c$ bounded edges, where $c$ is the number of $4$-valent vertices. By Lemma \ref{lem-g+1} and as this number has to be even because of the existence of an involution, $\Gamma$ has no or two $4$-valent vertices.
 
 If $\Gamma$ has no $4$-valent vertex, the quotient $\Gamma/\iota$
 has genus one. The multiplicity with which $\Gamma$ contributes to the count of twisted Hurwitz number equals (up to multiplication with the constant $2^{b}$ which does not depend on the choice of $\Gamma$) the multiplicity with which $\Gamma/\iota$ would contribute to the count of the double Hurwitz number counting branched covers of $\mathbb{P}^1$ with two special ramification profiles $\mu$ and $\nu$ and only simple ramification else. 
 We can thus view the piecewise polynomial of the twisted double Hurwitz number as a sum of (a constant times) the piecewise polynomial of the corresponding double Hurwitz number, plus a summand corresponding to the monodromy graphs which have two $4$-valent vertices. 
 By \cite{Joh15}, Corollary 1.2 (see Corollary \ref{cor-joh} above), the first summand is a polynomial whose homogeneous parts range from degree $\ell(\mu)+\ell(\nu)-1$ to $\ell(\mu)+\ell(\nu)+1$.
 Note that this result is non trivial and cannot be deduced immediately from the expression of the piecewise polynomial as a sum of polynomials for each monodromy graph: a single monodromy graph can have parts of lower degree, but in the total sum the parts of lower degree cancel, see Example 3.12 in \cite{CJM11}.
 
  Now assume $\Gamma$ has two $4$-valent vertices. 
 Then $\Gamma/\iota$ is of genus $0$, and its multiplicity is computed (up to constant) as $\prod_e \omega(e)$ times $(\omega_{V_1}-1)\cdot (\omega_{V_2}-1)$, where $V_1$ and $V_2$ denote the two $4$-valent vertices, $\omega_{V_i}$ the weight of their adjacent edges and $\omega(e)$ the weight of the edge $e$.
 As $\Gamma/\iota$ has $(\ell(\mu)+\ell(\nu))-1$ bounded edges, the first factor yields a homogeneous polynomial of degree $(\ell(\mu)+\ell(\nu))-1$. Multiplying with $(\omega_{V_1}-1)\cdot (\omega_{V_2}-1)$, we obtain a polynomial whose homogeneous parts range from degree $\ell(\mu)+\ell(\nu)-1$ to $\ell(\mu)+\ell(\nu)+1$.
 
 The two summands together, i.e.\ the contribution of graphs with no and with two $4$-valent vertices, then again yields a polynomial whose homogeneous parts range from degree $\ell(\mu)+\ell(\nu)-1$ to $\ell(\mu)+\ell(\nu)+1$.
 
\end{proof}

\begin{example}
In the following, we compute the polynomial $\tilde{h}_1((\mu),(\mu))$. As the partition $(\mu)$ consists of just one entry, there are no walls and there is no piecewise structure, but there is only one polynomial. The two monodromy graphs which contribute to the count of $\tilde{h}_1((\mu),(\mu))$ are depicted in Figure \ref{fig:exppgenus1}. The polynomial with which they contribute is written next to the picture.

Altogether, we obtain
\begin{align}
 \tilde{h}_1((\mu),(\mu))&= \frac{1}{2}\mu(\mu-1)^2+\mu\cdot \sum_{i=1}^{\mu-1} i - \sum_{i=1}^{\mu-1} i^2
 \\&= \frac{1}{2}\mu(\mu-1)^2+\frac{1}{2}\mu^2\cdot (\mu-1)-\frac{1}{6}\mu\cdot (\mu-1)\cdot (2\mu-1)\\
 & = \frac{2}{3}\mu^3-\mu^2+\frac{1}{3}\mu,
\end{align}
which is a polynomial whose homogeneous parts range from degree one to three, as expected by Proposition \ref{prop-ppingenus1}.

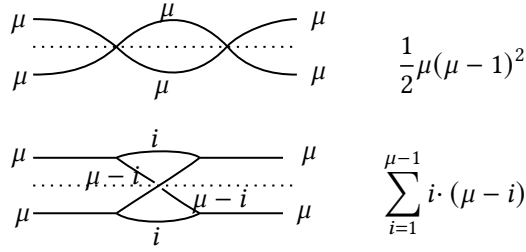
\begin{figure}
    \centering

\tikzset{every picture/.style={line width=0.75pt}} 

\begin{tikzpicture}[scale=0.7,x=0.75pt,y=0.75pt,yscale=-1,xscale=1]

\draw  [dash pattern={on 0.84pt off 2.51pt}]  (90,340) -- (280,340) ;
\draw    (90,320) .. controls (115.25,320.22) and (127.75,322.72) .. (150,340) ;
\draw    (90,360) .. controls (115.25,360.22) and (134.75,354.22) .. (150,340) ;
\draw    (150,340) .. controls (177,313.9) and (205.57,314.75) .. (229.86,339.61) ;
\draw    (150,340) .. controls (173.57,363.32) and (205.86,366.18) .. (229.86,339.61) ;
\draw    (229.86,339.61) .. controls (238.71,329.9) and (261.75,320.22) .. (280,320) ;
\draw    (229.86,339.61) .. controls (242.43,355.32) and (261.75,360.22) .. (280,360) ;
\draw  [dash pattern={on 0.84pt off 2.51pt}]  (90,440) -- (280,440) ;
\draw    (150,420) .. controls (161.25,414.22) and (202.14,413.32) .. (210,420) ;
\draw    (150,460) .. controls (160.75,467.72) and (194.71,468.47) .. (210,460) ;
\draw    (150,420) .. controls (160.75,427.72) and (165.29,430.18) .. (177.57,437.9) ;
\draw    (90,420) -- (150,420) ;
\draw    (90,460) -- (150,460) ;
\draw    (150,460) .. controls (162.75,451.22) and (196.25,430.72) .. (210,420) ;
\draw    (182.86,442.29) .. controls (193.61,450) and (197.29,452.75) .. (210,460) ;
\draw    (210,420) -- (270,420) ;
\draw    (210,460) -- (270,460) ;

\draw (72,314.29) node [anchor=north west][inner sep=0.75pt]   [align=left] {$\displaystyle \mu $};
\draw (73.14,355.71) node [anchor=north west][inner sep=0.75pt]   [align=left] {$\displaystyle \mu $};
\draw (179.14,306.57) node [anchor=north west][inner sep=0.75pt]   [align=left] {$\displaystyle \mu $};
\draw (175.29,362) node [anchor=north west][inner sep=0.75pt]   [align=left] {$\displaystyle \mu $};
\draw (72.29,411.43) node [anchor=north west][inner sep=0.75pt]   [align=left] {$\displaystyle \mu $};
\draw (74.86,454.57) node [anchor=north west][inner sep=0.75pt]   [align=left] {$\displaystyle \mu $};
\draw (125.57,424) node [anchor=north west][inner sep=0.75pt]   [align=left] {$\displaystyle \mu -i$};
\draw (350,330) node [anchor=north west][inner sep=0.75pt]   [align=left] {$\displaystyle \frac{1}{2}\mu ( \mu -1)^{2}$};
\draw (340.5,410.79) node [anchor=north west][inner sep=0.75pt]   [align=left] {$\displaystyle \sum _{i=1}^{\mu -1} i\cdotp ( \mu -i)$};
\draw (288,352) node [anchor=north west][inner sep=0.75pt]   [align=left] {$\displaystyle \mu $};
\draw (288,312) node [anchor=north west][inner sep=0.75pt]   [align=left] {$\displaystyle \mu $};
\draw (281.43,409.71) node [anchor=north west][inner sep=0.75pt]   [align=left] {$\displaystyle \mu $};
\draw (277.71,452.29) node [anchor=north west][inner sep=0.75pt]   [align=left] {$\displaystyle \mu $};
\draw (172.29,397.43) node [anchor=north west][inner sep=0.75pt]   [align=left] {$\displaystyle i$};
\draw (173.43,467.71) node [anchor=north west][inner sep=0.75pt]   [align=left] {$\displaystyle i$};
\draw (202.71,440.86) node [anchor=north west][inner sep=0.75pt]   [align=left] {$\displaystyle \mu -i$};

\end{tikzpicture}

    \caption{The computation of $\tilde{h}_1((\mu),(\mu))$.}
    \label{fig:exppgenus1}
\end{figure}

\end{example}

\begin{example}\label{ex-linearpart}
We now consider the polynomial $\tilde{h}_1(\mu,(\nu_1,\nu_2,\nu_3))$. By the same reasoning as in the previous example, there are no walls and no piecewise structure. We demonstrate that the lower bound on the degree given in \cref{prop-ppingenus1} involves nontrivial cancellations. For the Hurwitz number at hand, we obtain via \cref{prop-ppingenus0} that all monomials in the polynomial expression have degree at least $2$. In the following, we show that $\tilde{h}_1(\mu,(\nu_1,\nu_2,\nu_3))$ does not have any linear terms.\\
For any graph contributing to $\tilde{h}_1(\mu,(\nu_1,\nu_2,\nu_3))$, the polynomial weight is given as a sum of products of $\mu,\nu_i$ and Faulhaber sums $\sum_i i^k$ for some fixed positive integer $k$ and bounds given by linear forms in $\mu$ and the $\nu_i$. Therefore, in order to obtain a linear term, we would need a summand of $\tilde{h}_1(\mu,(\nu_1,\nu_2,\nu_3))$ to be of the form $\sum_i i^k$ for $k\ge 4$. The quotients by involution of the only four graphs (up to permutation of the $\nu_i$) contributing to $\tilde{h}_1(\mu,(\nu_1,\nu_2,\nu_3))$ are illustrated in \cref{fig:exampleg1}. We note that the vertices are ordered from left to right, which distinghuishes the top two graphs from each other.\\
The weight of the top left graph is given by
\begin{equation}
\label{equ:weightg1}
    16(\mu-i)i(\nu_1+\nu_2-i)(\nu_1-i),
\end{equation}
where $i$ lies in the interval $[\nu_1+1,\nu_1+\nu_2]$.
Since, we are only interested in the linear part, we will focus on the contribution of $-i^4$ and we obtain
\begin{equation}
    \textrm{linear part of}\,-\sum_{i=\nu_1}^{\nu_1+\nu_2}i^4=-\frac{\nu_2}{30}.
\end{equation}
Thus, the graph on the top left contributes a linear term of $-16\frac{\nu_2}{30}$.
The same calculation for the top right graphs yields a contribution to the linear term of $-16\frac{\nu_2}{30}$ as well.\\
The weight of the bottom left graph is given by
\begin{equation}
    16(\mu-i)i(i-\nu_1)(i-\nu_1-\nu_2),
\end{equation}
where $i$ lies in the interval $[\nu_1+\nu_2+1,\mu]$. For the linear part, we focus on the contribution of $-i^4$ which yields a total contribution to the linear term of $16\frac{\nu_3}{30}$.\\
Finally, the graph on the bottom right has weight
\begin{equation}
    16(\mu-i)i(\nu_1-i)(\nu_1+\nu_2-i)
\end{equation}
where $i$ lies in the interval $[1,\nu_1]$. Here, the sum over $-i^4$ yields a total contribution of $16\frac{\nu_1}{30}$ to the linear term.\\
Thus, we obtain a linear term of
\begin{equation}
    16(\frac{\nu_1+\nu_3}{30}-\frac{\nu_2}{15})
\end{equation}
summing over all four graphs. Note that the four graphs in \cref{fig:exampleg1} all have $\nu_2$ as the middle strand. Permuting the last strands, such that $\nu_1$ and $\nu_3$ replace $\nu_2$ yields a total linear term of
\begin{equation}
     16(\frac{\nu_1+\nu_3}{30}-\frac{\nu_2}{15})+ 16(\frac{\nu_2+\nu_3}{30}-\frac{\nu_1}{15})+ 16(\frac{\nu_1+\nu_2}{30}-\frac{\nu_3}{15})=0.
\end{equation}

Thus, all linear contributions cancel and indeed the polynomial $\tilde{h}_1(\mu,(\nu_1,\nu_2,\nu_3))$ has no linear monomials.






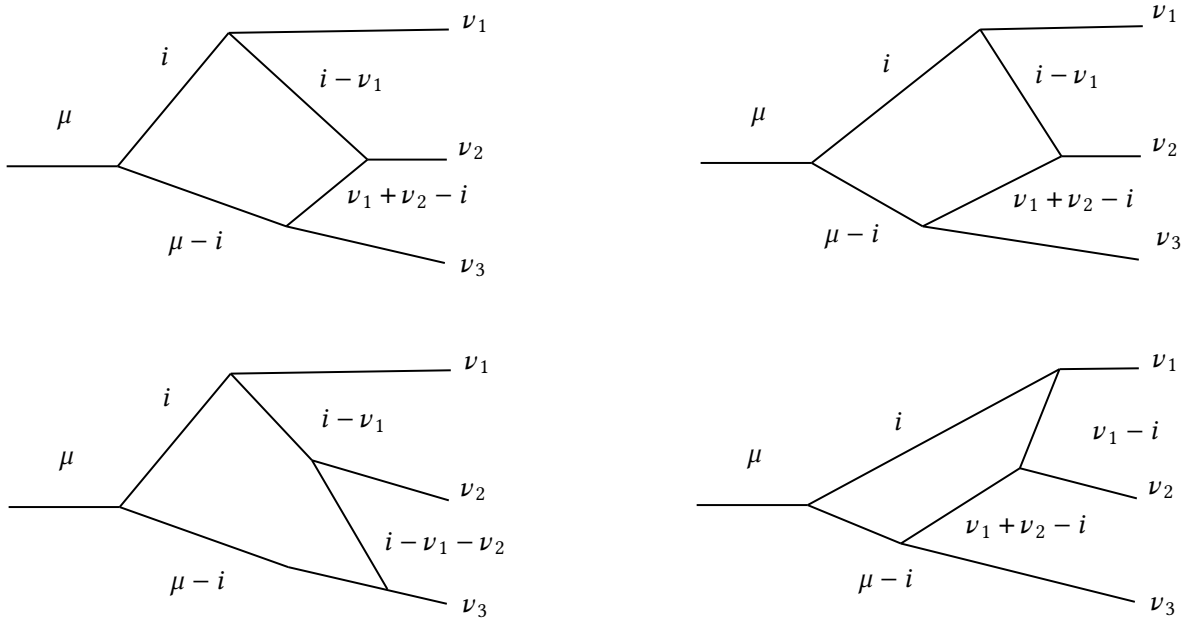
\begin{figure}

\tikzset{every picture/.style={line width=0.75pt}} 

\begin{tikzpicture}[x=0.75pt,y=0.75pt,yscale=-1,xscale=1]

\draw    (29,96.6) -- (85,96.6) ;
\draw    (85,96.6) -- (141,29.26) ;
\draw    (141,29.26) -- (211,93.23) ;
\draw    (141,29.26) -- (252,27.57) ;
\draw    (85,96.6) -- (170,126.91) ;
\draw    (170,126.91) -- (211,93.23) ;
\draw    (211,93.23) -- (251,93.23) ;
\draw    (170,126.91) -- (250,145.43) ;
\draw    (379,94.92) -- (435,94.92) ;
\draw    (435,94.92) -- (520,27.57) ;
\draw    (520,27.57) -- (561,91.55) ;
\draw    (520,27.57) -- (602,25.89) ;
\draw    (435,94.92) -- (491,126.91) ;
\draw    (491,126.91) -- (561,91.55) ;
\draw    (561,91.55) -- (601,91.55) ;
\draw    (491,126.91) -- (600,143.74) ;
\draw    (30,268.6) -- (86,268.6) ;
\draw    (86,268.6) -- (142,201.26) ;
\draw    (142,201.26) -- (183,245) ;
\draw    (142,201.26) -- (253,199.57) ;
\draw    (86,268.6) -- (171,298.91) ;
\draw    (221,310) -- (183,245) ;
\draw    (183,245) -- (252,265.23) ;
\draw    (171,298.91) -- (251,317.43) ;
\draw    (377,267.6) -- (433,267.6) ;
\draw    (433,267.6) -- (560,199) ;
\draw    (560,199) -- (540,249) ;
\draw    (560,199) -- (600,198.57) ;
\draw    (433,267.6) -- (480,287) ;
\draw    (480,287) -- (540,249) ;
\draw    (540,249) -- (599,264.23) ;
\draw    (480,287) -- (598,316.43) ;

\draw (53,67.11) node [anchor=north west][inner sep=0.75pt]    {$\mu $};
\draw (403,65.43) node [anchor=north west][inner sep=0.75pt]    {$\mu $};
\draw (606,13.92) node [anchor=north west][inner sep=0.75pt]    {$\nu _{1}$};
\draw (258,18.97) node [anchor=north west][inner sep=0.75pt]    {$\nu _{1}$};
\draw (608,128.41) node [anchor=north west][inner sep=0.75pt]    {$\nu _{3}$};
\draw (606,81.27) node [anchor=north west][inner sep=0.75pt]    {$\nu _{2}$};
\draw (256,82.95) node [anchor=north west][inner sep=0.75pt]    {$\nu _{2}$};
\draw (257,141.88) node [anchor=north west][inner sep=0.75pt]    {$\nu _{3}$};
\draw (105,35.12) node [anchor=north west][inner sep=0.75pt]    {$i$};
\draw (469,38.49) node [anchor=north west][inner sep=0.75pt]    {$i$};
\draw (109,127.73) node [anchor=north west][inner sep=0.75pt]    {$\mu -i$};
\draw (440,124.36) node [anchor=north west][inner sep=0.75pt]    {$\mu -i$};
\draw (185,45.91) node [anchor=north west][inner sep=0.75pt]    {$i-\nu _{1}$};
\draw (546,44.23) node [anchor=north west][inner sep=0.75pt]    {$i-\nu _{1}$};
\draw (200,104.84) node [anchor=north west][inner sep=0.75pt]    {$\nu _{1} +\nu _{2} -i$};
\draw (536,106.52) node [anchor=north west][inner sep=0.75pt]    {$\nu _{1} +\nu _{2} -i$};
\draw (54,239.11) node [anchor=north west][inner sep=0.75pt]    {$\mu $};
\draw (259,190.97) node [anchor=north west][inner sep=0.75pt]    {$\nu _{1}$};
\draw (257,254.95) node [anchor=north west][inner sep=0.75pt]    {$\nu _{2}$};
\draw (258,313.88) node [anchor=north west][inner sep=0.75pt]    {$\nu _{3}$};
\draw (106,207.12) node [anchor=north west][inner sep=0.75pt]    {$i$};
\draw (110,299.73) node [anchor=north west][inner sep=0.75pt]    {$\mu -i$};
\draw (186,217.91) node [anchor=north west][inner sep=0.75pt]    {$i-\nu _{1}$};
\draw (218,279.84) node [anchor=north west][inner sep=0.75pt]    {$i-\nu _{1} -\nu _{2}$};
\draw (401,238.11) node [anchor=north west][inner sep=0.75pt]    {$\mu $};
\draw (606,189.97) node [anchor=north west][inner sep=0.75pt]    {$\nu _{1}$};
\draw (604,253.95) node [anchor=north west][inner sep=0.75pt]    {$\nu _{2}$};
\draw (605,312.88) node [anchor=north west][inner sep=0.75pt]    {$\nu _{3}$};
\draw (475,219.12) node [anchor=north west][inner sep=0.75pt]    {$i$};
\draw (457,298.73) node [anchor=north west][inner sep=0.75pt]    {$\mu -i$};
\draw (576,223.91) node [anchor=north west][inner sep=0.75pt]    {$\nu _{1} -i$};
\draw (512,271.4) node [anchor=north west][inner sep=0.75pt]    {$\nu _{1} +\nu _{2} -i$};

\end{tikzpicture}
     \caption{Quotients of monodromy graphs contributing to $\tilde{h}_1(\mu,(\nu_1,\nu_2,\nu_3))$, see Example \ref{ex-linearpart}.}
     \label{fig:exampleg1}
\end{figure}

\end{example}

\section{Towards wall crossing formulae for twisted double Hurwitz numbers in genus zero}

We investigate how the polynomials computing twisted Hurwitz numbers vary from chamber to chamber. 
We build on the techniques to prove wall-crossing formulae in genus $0$ for double Hurwitz numbers developed in \cite{CJM10} (see also \cite{SSV08}).

\begin{definition}[Wall-crossing]
Fix a wall $\delta=\sum_{i\in I}\mu_i - \sum_{j\in J}\nu_j=0$ and two adjacent chambers $C_1$ and $C_2$. Let $P_1$ denote the polynomial that equals $\tilde{h}_0(\mu,\nu)$ in $C_1$, $P_2$ the polynomial that equals $\tilde{h}_0(\mu,\nu)$ in $C_2$.
Then the wall crossing for this wall and the two adjacent chambers $C_1$ and $C_2$ is defined to be the difference of the two polynomials:

\begin{equation}
\WC_{\delta}(\mu,\nu):=P_1(\mu,\nu)-P_2(\mu,\nu).\end{equation}

\end{definition}
 
Since we are only concerned with the difference of the polynomials across a wall $\delta=0$, we need only consider the contributions from graphs that contribute to the twisted Hurwitz number in only one of the two chambers in questions. These are precisely the graphs that  contain an edge with weight $\delta$ that switches direction across the wall (see Lemma 6.9 \cite{CJM10}). 

We cut this edge. In this way, we obtain two connected components, each contributing to a ''smaller'' Hurwitz number.

Vice versa, two graphs that contribute to these ''smaller'' Hurwitz numbers can be glued to produce one contributing to the wall crossing.

Compared to the case of double Hurwitz numbers studied in \cite{CJM10}, the only difference is that for twisted Hurwitz numbers of genus zero, our quotient graphs $\Gamma/\iota$ contain precisely one $2$-valent vertex. This vertex, adjacent to an edge of weight $m$, contributes an additional factor of $m(m-1)$ to the weight of the graph compared to the weight that this graph would have without the $2$-valent vertex and contributing to a usual double Hurwitz number.

Accordingly, we obtain three summands in our wall-crossing formula, which correspond to the following three cases:

\begin{itemize}
    \item The $2$-valent vertex is in the first part of the cut graph.
    \item The $2$-valent vertex is in the second part of the cut graph.
    \item The $2$-valent vertex is on the edge which we cut.
\end{itemize}

Due the weight contributed by a $2$-valent vertex, the techniques from \cite{CJM10} do not carry over completely and therefore our description of wall crossings involves a correction term that cannot be described in terms of (twisted) Hurwitz numbers. Thus, we will need the following definition.

\begin{definition}
We fix a wall $\delta=\sum_{i\in I}\mu_i - \sum_{j\in J}\nu_j=0$ and fix an adjacent chamber $C$. Then, we define \begin{equation}
    h_0^{C,\delta}(\mu,\nu)=\sum_{\Gamma} o(\Gamma) \cdot 2^{b}\cdot \prod_V (\omega_V-1)\prod_e\omega(e)\cdot \frac{1}{|\Aut(\pi)|},
\end{equation}
where  the sum goes over all twisted monodromy graphs contributing to $\tilde{h}_0(\mu,\nu)$ in the chamber $C$ without a fixed vertex ordering containing an edge of weight $\delta$ that is adjacent to a $2$-valent vertex. Furthermore, the number of vertex orderings which are compatible with the edge directions is denoted by $o(\Gamma)$.
\end{definition}

By the same arguments as used in the proof of \cref{thm-PP}, we have that $h_0^{C,\delta}(\mu,\nu)$ is a polynomial in the entries of $\mu$ and $\nu$.

\begin{proposition}[Wall-crossing for twisted double Hurwitz numbers in genus zero]
\label{thm-wc}

Fix a wall $\delta=\sum_{i\in I}\mu_i - \sum_{j\in J}\nu_j=0$.
Then the wall crossing for this wall and the two adjacent chambers $C_1$ and $C_2$ equals a sum of products of ''smaller'', traditional and twisted, Hurwitz numbers with an additional correction term:

\begin{align}\WC_{\delta}(\mu,\nu)= &\big(h_0^{C_1,\delta}(\mu,\nu)- h_0^{C_2,\delta}(\mu,\nu)\big)+
  \delta\cdot \\
&\Big(2^{|I^c|+|J^c|-1}\cdot\binom{\ell(\mu)+\ell(\nu)-1}{\ell(\mu_I)+\ell(\nu_J)}\cdot \tilde{h}_0\big(\mu_I,(\nu_J,\delta)\big)\cdot h_0\big((\mu_{I^c},\delta),\nu_{J^c}\big)
\\&+2^{|I|+|J|-1}\cdot\binom{\ell(\mu)+\ell(\nu)-1}{\ell(\mu_I)+\ell(\nu_J)-1}\cdot h_0\big(\mu_I,(\nu_J,\delta)\big)\cdot \tilde{h}_0\big((\mu_{I^c},\delta),\nu_{J^c}\big)\Big).
\end{align}
Here, $h_0(\mu,\nu)$ denotes the double Hurwitz number counting covers of the projective line with fixed ramification profile $\mu$ over $0$, $\nu$ over $\infty$ and simple ramification over $r$ additional fixed branch points.
\end{proposition}

Before we start with the proof, we collect the following necessary statements:

\begin{lemma}\label{lem-compareweights}
Let $\Gamma$ be a twisted monodromy graph of type $(0,\mu,\nu)$ and let $p$ denote the weight with which it contributes to $\tilde{h}_0(\mu,\nu)$. 
Then the graph $\tilde{\Gamma}$ we obtain from $\Gamma/\iota$ by forgetting the $2$-valent vertex contributes to the Hurwitz number $h_0(\mu,\nu)$ with weight
$$\frac{1}{m\cdot(m-1)} \cdot \frac{1}{2^{b-2}}\cdot p,$$
where $m$ denotes the weight of the edges adjacent to the $4$-valent vertex of $\Gamma$.

\end{lemma}
\begin{proof}
By Proposition \ref{prop-ppingenus0}, $\Gamma$ has precisely one $4$-valent vertex. Accordingly, $\Gamma/\iota$ has precisely one $2$-valent vertex and thus $\tilde{\Gamma}$ is well-defined.

Assume the $4$-valent vertex is adjacent to edges of weight $m$.

The graph $\tilde{\Gamma}$ is $3$-valent and has in-ends labeled by the $\mu_i$ and out-ends labeled by the $\nu_j$. By Corollary 4.4 in \cite{CJM10}, it contributes the product of the weights of its bounded edges towards the Hurwitz number $h_0(\mu,\nu)$. Notice that a $3$-valent tropical cover of genus $0$ with different branch points cannot have non-trivial automorphisms if the weights of the ends are viewed as variable, as balanced forks (i.e. adjacent ends of the same weight) are the only source of automorphism.

By Corollary \ref{thm-monodromygraphs}, the weight $p$ with which $\Gamma$ contributes of $\tilde{h}_0(\mu,\nu)$ equals 

$2^{b}\cdot (m-1)\cdot\prod_e\omega(e)\cdot \frac{1}{|\Aut(\pi)|}$.

Here, $b$ is the number of branch points, and the product goes over all bounded edges of $\Gamma/\iota$.


The involution $\iota$ of $\Gamma$ provides a non-trivial automorphism. By Proposition \ref{prop-ppingenus0}, $\Gamma$ has precisely one $4$-valent vertex. This provides an additional automorphism. As the quotient with respect to $i$ does not have further automorphisms, we conclude that $|\Aut(\pi)|=4$.

The graph $\Gamma/\iota$ has an additional bounded edge of weight $m$ compared to $\tilde{\Gamma}$, this edge gets lost when we forget the $2$-valent vertex.

It follows that $p$ equals
$2^{b-2}\cdot m \cdot (m-1)$ times the weight with which $\tilde{\Gamma}$ contributes to $h_0(\mu,\nu)$.

\end{proof}

\begin{remark}
As a vertex ordering of $\Gamma/\iota$ implies a vertex ordering of $\Gamma$ and vice versa, Lemma \ref{lem-compareweights} can be extended to monodromy graphs without a fixed vertex ordering. If we fix a vertex ordering on $\tilde{\Gamma}$ and a position for the $2$-valent vertex, the vertex ordering on $\Gamma$ is given.
\end{remark}

\begin{proof}[Proof of Proposition \ref{thm-wc}:]
By Lemma 6.9 of \cite{CJM10}, only graphs that have an edge of weight $\delta$ contribute to the wall-crossing.

Let $\Gamma$ be a twisted monodromy graph of type $(0,\mu,\nu)$ without a vertex ordering with a pair of bounded edges of weight $\delta$.

We can express the weight with which $\Gamma$ contributes in terms of the weight $m$ of the edge adjacent to the $2$-valent vertex and the weight of $\tilde{\Gamma}$, as in Lemma \ref{lem-compareweights}.

The graph $\tilde{\Gamma}$ has precisely one edge of weight $\delta$. This holds true, since in genus $0$, the weight of each edge is given, by the balancing condition, as the sum of the (signed) weights of the ends which get separated by cutting an edge.  
On the other side of the wall, the graph but with the direction of the edge reversed contributes. The first term corresponds to the contribution of those graphs where the edge of weight $\delta$ is adjacent to a $2$-valent vertex. 

We now assume that the edge of weight $\delta$ of the graph $\tilde{\Gamma}$ is not adjacent to a $2$-valent vertex. We express the difference of those two contributions (i.e.\ the contribution to the wall-crossing governed by $\Gamma$) in terms of cut graphs.

We cut the edge $\delta$ of $\tilde{\Gamma}$.
Then we obtain two connected components, one containing the in-ends with weights $\mu_i$ where $i\in I$ and the out-ends with weights $\nu_j$, $j\in J$, plus an additional new end that we view as an out-end of weight $\delta$, and the other one containing the in-ends $\mu_i$ for $i \notin I$, with an additional in-end of weight $\delta$, and the out-ends $\nu_j$ for $j\notin J$.

These graphs contribute to ''smaller'' Hurwitz numbers $h_0\big(\mu_I,(\nu_J,\delta)\big)$ and $h_0\big((\mu_{I^c},\delta),\nu_{J^c}\big)$.

But we still need to take the $4$-valent vertex into account.
That is, we now mark the $2$-valent vertex in the cut graph $\tilde{\Gamma}$ again and
distinguish two cases according to its position:
\begin{itemize}
    \item either the $2$-valent vertex belongs to the component contributing to $h_0\big(\mu_I,(\nu_J,\delta)\big)$, 
    \item or to $h_0\big((\mu_{I^c},\delta),\nu_{J^c}\big)$.
\end{itemize}

In both cases, we can interpret the piece with the $2$-valent vertex as a contribution to a twisted Hurwitz number.

Vice versa, if two pieces are given, we can re-glue to obtain $\Gamma/\iota$. 
The piece with ends $(\mu_I,(\nu_J,\delta))$ (we call it the first piece) has $\ell(\mu_I)+\ell(\nu_J)+1$ ends, and thus, by an Euler characteristics computation, $\ell(\mu_I)+\ell(\nu_J)-1$ threevalent vertices. In the first summand, it also has the $2$-valent vertex in addition.
The whole graph has $\ell(\mu)+\ell(\nu)-1$ vertices, and the second piece (with ends  $((\mu_{I^c},\delta),\nu_{J^c})$)
has $\ell(\mu_{I^c})+\ell(\nu_{J^c})-1$ vertices. 
The binomial factor $\binom{\ell(\mu)+\ell(\nu)-1}{\ell(\mu_I)+\ell(\nu_J)}$ in the first summand accounts for the possibilities to pick positions for the $\ell(\mu_I)+\ell(\nu_J)$ vertices of the first piece among the total number of $\ell(\mu)+\ell(\nu)-1$ branch points. Since the vertices in both pieces are ordered among themselves, each such choice yields an ordering of all the vertices. Some of these choices correspond to orderings for which the edge with weight $\delta$ is oriented in one way, and some to orderings for which it is oriented in the other way. Accordingly, the glued picture can belong to one side of the wall or to the other. But as $\delta$ changes sign as we cross the wall (and as we subtract the contribution on one side from the other), the sign is taken care of automatically and the glued graph contributes correctly to the wall-crossing.
In the second summand, the binomial factor has to be changed to $\binom{\ell(\mu)+\ell(\nu)-1}{\ell(\mu_I)+\ell(\nu_J)-1}$, as there the $2$-valent vertex is in the second piece.
\end{proof}

To end with, we demonstrate in an example why we excluded the case of the edge of weight $\delta$ being adjacent to a $2$-valent vertex in the discussion on cut-graphs in the proof of \cref{thm-wc}.

\begin{example}
We consider the twisted Hurwitz number $\tilde{h}_0((\mu,\nu),(\lambda,\kappa))$ and fix $\delta=\mu-\lambda=\kappa-\nu$. Moreover, we consider the graph $\Gamma$ in the top of \cref{fig:badex} contributing to it. In the bottom, the graph $\tilde{\Gamma}$ illustrates its quotient by the indicated involution. We observe that there are only two possible vertex orderings of $\tilde{\Gamma}$, the first being the one illustrated and the second being the inverse one. First, we observe that in the chamber $\delta>0$ the given graph is counted with the weight
\begin{equation}
    2^3\cdot (\delta-1)\delta^2.
\end{equation}
However, in the chamber $\delta<0$, the graph contributes a weight of
\begin{equation}
        2^3\cdot (-\delta-1)\delta^2.
\end{equation}

Therefore, factoring out the edge weights is not possible in this situation, which represents the first obstacle for the tropical approach of \cite{CJM11}.\\
The second obstacle arises from the possible number of orientations after gluing together the cut-graphs. This is illustrated in \cref{fig:badexcut}. In particular, we observe that the $2$-valent vertex imposes the condition that it lies between the upper and the lower vertex of the respective cut-graphs. The enumeration of possible orderings is therefore not a simple binomial coefficient and depends on the graph at hand.
\end{example}
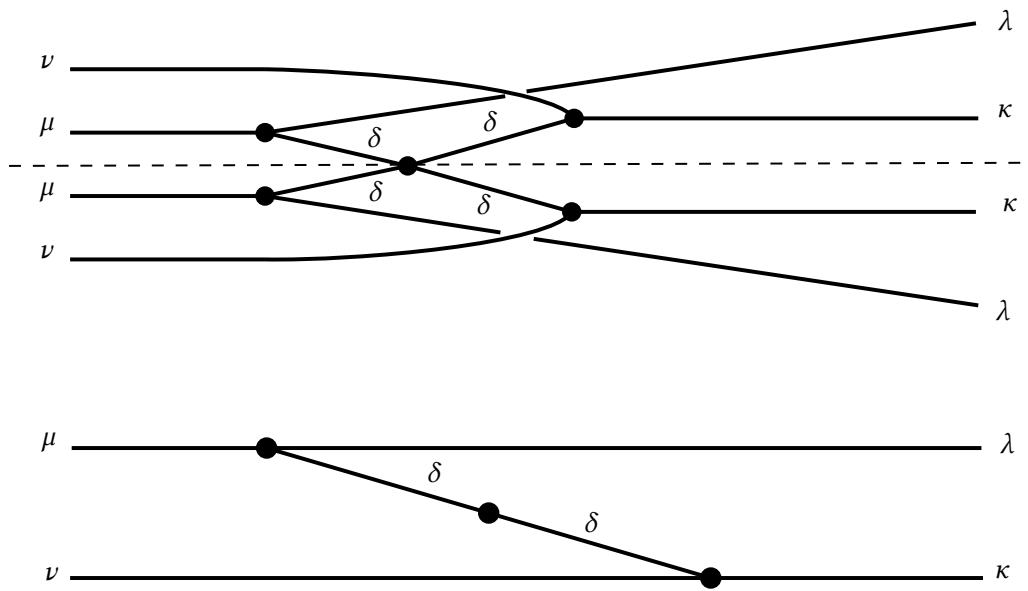
\begin{figure}

\tikzset{every picture/.style={line width=0.75pt}} 

\begin{tikzpicture}[scale=0.8,x=0.75pt,y=0.75pt,yscale=-1,xscale=1]

\draw [line width=1.5]    (47.48,50) -- (170.38,50) ;
\draw [line width=1.5]    (47.48,90) -- (170.38,90) ;
\draw [line width=1.5]    (47.48,130) -- (170.38,130) ;
\draw [line width=1.5]    (47.48,170) -- (170.38,170) ;
\draw [line width=1.5]    (170.38,90) -- (260.31,111) ;
\draw [line width=1.5]    (170.38,130) -- (260.31,111) ;
\draw [line width=1.5]    (260.31,111) -- (365.23,81) ;
\draw [line width=1.5]    (365.23,81) -- (620.02,81) ;
\draw [line width=1.5]    (260.31,111) -- (363.73,140) ;
\draw [line width=1.5]    (170.38,90) -- (321.76,67) ;
\draw [line width=1.5]    (335.25,64) -- (618.53,21) ;
\draw [line width=1.5]    (363.73,140) -- (618.53,140) ;
\draw [line width=1.5]    (170.38,130) -- (318.76,153) ;
\draw [line width=1.5]    (339.75,157) -- (620.02,199) ;
\draw [line width=1.5]    (170.38,50) .. controls (228.84,51) and (351.74,61) .. (365.23,81) ;
\draw [line width=1.5]    (170.38,170) .. controls (228.84,171) and (351.74,160) .. (363.73,140) ;
\draw  [color={rgb, 255:red, 0; green, 0; blue, 0 }  ,draw opacity=1 ][fill={rgb, 255:red, 0; green, 0; blue, 0 }  ,fill opacity=1 ] (164.88,130) .. controls (164.88,126.96) and (167.34,124.5) .. (170.38,124.5) .. controls (173.42,124.5) and (175.88,126.96) .. (175.88,130) .. controls (175.88,133.04) and (173.42,135.5) .. (170.38,135.5) .. controls (167.34,135.5) and (164.88,133.04) .. (164.88,130) -- cycle ;
\draw  [color={rgb, 255:red, 0; green, 0; blue, 0 }  ,draw opacity=1 ][fill={rgb, 255:red, 0; green, 0; blue, 0 }  ,fill opacity=1 ] (164.88,90) .. controls (164.88,86.96) and (167.34,84.5) .. (170.38,84.5) .. controls (173.42,84.5) and (175.88,86.96) .. (175.88,90) .. controls (175.88,93.04) and (173.42,95.5) .. (170.38,95.5) .. controls (167.34,95.5) and (164.88,93.04) .. (164.88,90) -- cycle ;
\draw  [color={rgb, 255:red, 0; green, 0; blue, 0 }  ,draw opacity=1 ][fill={rgb, 255:red, 0; green, 0; blue, 0 }  ,fill opacity=1 ] (254.81,111) .. controls (254.81,107.96) and (257.27,105.5) .. (260.31,105.5) .. controls (263.35,105.5) and (265.81,107.96) .. (265.81,111) .. controls (265.81,114.04) and (263.35,116.5) .. (260.31,116.5) .. controls (257.27,116.5) and (254.81,114.04) .. (254.81,111) -- cycle ;
\draw  [color={rgb, 255:red, 0; green, 0; blue, 0 }  ,draw opacity=1 ][fill={rgb, 255:red, 0; green, 0; blue, 0 }  ,fill opacity=1 ] (358.23,140) .. controls (358.23,136.96) and (360.69,134.5) .. (363.73,134.5) .. controls (366.77,134.5) and (369.23,136.96) .. (369.23,140) .. controls (369.23,143.04) and (366.77,145.5) .. (363.73,145.5) .. controls (360.69,145.5) and (358.23,143.04) .. (358.23,140) -- cycle ;
\draw  [color={rgb, 255:red, 0; green, 0; blue, 0 }  ,draw opacity=1 ][fill={rgb, 255:red, 0; green, 0; blue, 0 }  ,fill opacity=1 ] (359.73,81) .. controls (359.73,77.96) and (362.19,75.5) .. (365.23,75.5) .. controls (368.26,75.5) and (370.73,77.96) .. (370.73,81) .. controls (370.73,84.04) and (368.26,86.5) .. (365.23,86.5) .. controls (362.19,86.5) and (359.73,84.04) .. (359.73,81) -- cycle ;
\draw  [dash pattern={on 4.5pt off 4.5pt}]  (9,111) -- (657,109) ;
\draw [line width=1.5]    (48.48,289) -- (171.38,289) ;
\draw  [color={rgb, 255:red, 0; green, 0; blue, 0 }  ,draw opacity=1 ][fill={rgb, 255:red, 0; green, 0; blue, 0 }  ,fill opacity=1 ][line width=1.5]  (165.88,289) .. controls (165.88,285.96) and (168.34,283.5) .. (171.38,283.5) .. controls (174.42,283.5) and (176.88,285.96) .. (176.88,289) .. controls (176.88,292.04) and (174.42,294.5) .. (171.38,294.5) .. controls (168.34,294.5) and (165.88,292.04) .. (165.88,289) -- cycle ;
\draw [line width=1.5]    (171.38,289) -- (622,289) ;
\draw [line width=1.5]    (47.48,371) -- (623,371) ;
\draw  [color={rgb, 255:red, 0; green, 0; blue, 0 }  ,draw opacity=1 ][fill={rgb, 255:red, 0; green, 0; blue, 0 }  ,fill opacity=1 ][line width=1.5]  (445.88,371) .. controls (445.88,367.96) and (448.34,365.5) .. (451.38,365.5) .. controls (454.42,365.5) and (456.88,367.96) .. (456.88,371) .. controls (456.88,374.04) and (454.42,376.5) .. (451.38,376.5) .. controls (448.34,376.5) and (445.88,374.04) .. (445.88,371) -- cycle ;
\draw [line width=1.5]    (171.38,289) -- (451.38,371) ;
\draw  [color={rgb, 255:red, 0; green, 0; blue, 0 }  ,draw opacity=1 ][fill={rgb, 255:red, 0; green, 0; blue, 0 }  ,fill opacity=1 ][line width=1.5]  (305.88,330) .. controls (305.88,326.96) and (308.34,324.5) .. (311.38,324.5) .. controls (314.42,324.5) and (316.88,326.96) .. (316.88,330) .. controls (316.88,333.04) and (314.42,335.5) .. (311.38,335.5) .. controls (308.34,335.5) and (305.88,333.04) .. (305.88,330) -- cycle ;

\draw (26.49,78.4) node [anchor=north west][inner sep=0.75pt]    {$\mu $};
\draw (26.49,119.4) node [anchor=north west][inner sep=0.75pt]    {$\mu $};
\draw (27.99,159.4) node [anchor=north west][inner sep=0.75pt]    {$\nu $};
\draw (27.99,39.4) node [anchor=north west][inner sep=0.75pt]    {$\nu $};
\draw (630.76,10.4) node [anchor=north west][inner sep=0.75pt]    {$\lambda $};
\draw (629.26,193.4) node [anchor=north west][inner sep=0.75pt]    {$\lambda $};
\draw (630.51,70.4) node [anchor=north west][inner sep=0.75pt]    {$\kappa $};
\draw (633.51,129.4) node [anchor=north west][inner sep=0.75pt]    {$\kappa $};
\draw (233,84.4) node [anchor=north west][inner sep=0.75pt]    {$\delta $};
\draw (234,118.4) node [anchor=north west][inner sep=0.75pt]    {$\delta $};
\draw (306,74.4) node [anchor=north west][inner sep=0.75pt]    {$\delta $};
\draw (302,126.4) node [anchor=north west][inner sep=0.75pt]    {$\delta $};
\draw (27.49,277.4) node [anchor=north west][inner sep=0.75pt]    {$\mu $};
\draw (632,277.4) node [anchor=north west][inner sep=0.75pt]    {$\lambda $};
\draw (30.99,362.4) node [anchor=north west][inner sep=0.75pt]    {$\nu $};
\draw (629,361.4) node [anchor=north west][inner sep=0.75pt]    {$\kappa $};
\draw (271,295.4) node [anchor=north west][inner sep=0.75pt]    {$\delta $};
\draw (370,327.4) node [anchor=north west][inner sep=0.75pt]    {$\delta $};

\end{tikzpicture}
    \caption{The graph on the top contributes to $\tilde{h}_{0}((\mu,\nu),(\lambda,\kappa))$. The graph on the bottom is the quotient by the involution.}
    \label{fig:badex}
\end{figure}

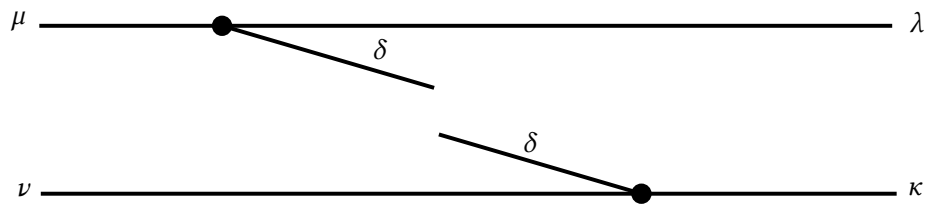
\begin{figure}

\tikzset{every picture/.style={line width=0.75pt}} 

\begin{tikzpicture}[scale=0.75,x=0.75pt,y=0.75pt,yscale=-1,xscale=1]

\draw [line width=1.5]    (30.99,59) -- (153.89,59) ;
\draw  [color={rgb, 255:red, 0; green, 0; blue, 0 }  ,draw opacity=1 ][fill={rgb, 255:red, 0; green, 0; blue, 0 }  ,fill opacity=1 ][line width=1.5]  (148.39,59) .. controls (148.39,55.96) and (150.85,53.5) .. (153.89,53.5) .. controls (156.93,53.5) and (159.39,55.96) .. (159.39,59) .. controls (159.39,62.04) and (156.93,64.5) .. (153.89,64.5) .. controls (150.85,64.5) and (148.39,62.04) .. (148.39,59) -- cycle ;
\draw [line width=1.5]    (153.89,59) -- (604.51,59) ;
\draw [line width=1.5]    (153.89,59) -- (296.48,100.76) ;
\draw [line width=1.5]    (31.99,172) -- (607.51,172) ;
\draw  [color={rgb, 255:red, 0; green, 0; blue, 0 }  ,draw opacity=1 ][fill={rgb, 255:red, 0; green, 0; blue, 0 }  ,fill opacity=1 ][line width=1.5]  (430.39,172) .. controls (430.39,168.96) and (432.85,166.5) .. (435.89,166.5) .. controls (438.93,166.5) and (441.39,168.96) .. (441.39,172) .. controls (441.39,175.04) and (438.93,177.5) .. (435.89,177.5) .. controls (432.85,177.5) and (430.39,175.04) .. (430.39,172) -- cycle ;
\draw [line width=1.5]    (299.68,132.11) -- (435.89,172) ;

\draw (10,47.4) node [anchor=north west][inner sep=0.75pt]    {$\mu $};
\draw (614.51,47.4) node [anchor=north west][inner sep=0.75pt]    {$\lambda $};
\draw (253.51,65.4) node [anchor=north west][inner sep=0.75pt]    {$\delta $};
\draw (15.5,163.4) node [anchor=north west][inner sep=0.75pt]    {$\nu $};
\draw (613.51,162.4) node [anchor=north west][inner sep=0.75pt]    {$\kappa $};
\draw (354.51,128.4) node [anchor=north west][inner sep=0.75pt]    {$\delta $};

\end{tikzpicture}
    \caption{The graph at the bottom of \cref{fig:badex} cut along the edge $\delta$.}
    \label{fig:badexcut}
\end{figure}

\printbibliography

\end{document}